\documentclass[10pt]{amsart}
\usepackage{amssymb}

\newtheorem{thm}{Theorem}[section]
\newtheorem{lemma}[thm]{Lemma}
\newtheorem{proposition}[thm]{Proposition}

\newtheorem{definition}[thm]{Definition}
\newtheorem{corollary}[thm]{Corollary}

\newtheorem{assumption}[thm]{Assumption}
\newtheorem*{maintheorem}{Main Theorem}

\newcommand{\p}{\mathbb{P}}
\newcommand{\q}{\mathbb{Q}}

\newcommand{\con}{\ \widehat{} \ }

\newcommand{\cf}{\mathrm{cf}}
\newcommand{\cof}{\mathrm{cof}}
\newcommand{\dom}{\mathrm{dom}}
\newcommand{\ran}{\mathrm{ran}}

\newcommand{\add}{\textrm{Add}}
\newcommand{\coll}{\textrm{Col}}

\newcommand{\lh}{\mathrm{lh}}
\newcommand{\suc}{\mathrm{Suc}}

\begin{document}

\title{Namba forcing, weak approximation, and guessing}

\author{Sean Cox and John Krueger}

\address{Sean Cox \\ Department of Mathematics and Applied Mathematics \\
Virginia Commonwealth University \\ 
1015 Floyd Avenue \\ 
PO Box 842014 \\ 
Richmond, Virginia 23284}
\email{scox9@vcu.edu}

\address{John Krueger \\ Department of Mathematics \\ 
University of North Texas \\
1155 Union Circle \#311430 \\
Denton, TX 76203}
\email{jkrueger@unt.edu}

\date{September 2016; revised May 2018}



\thanks{The work on this paper began while the authors attended the 
``High and low forcing'' workshop at the American Institute of 
Mathematics in January 2016. 
The first author was partially supported by the 
Simons Foundation grant 318467 
and the VCU Presidential Research Quest Fund. 
The second author was partially supported by the National 
Science Foundation Grant No.\ DMS-1464859.}

\begin{abstract}
We prove a variation of Easton's lemma for strongly proper forcings, 
and use it to prove that, unlike the stronger principle 
\textsf{IGMP}, \textsf{GMP} together with $2^\omega \le \omega_2$ 
is consistent with the existence of an 
$\omega_1$-distributive nowhere c.c.c.\! forcing poset of size $\omega_1$. 
We introduce the idea of a weakly guessing model, 
and prove that many of the strong consequences of the principle 
$\textsf{GMP}$ follow from the existence of stationarily many 
weakly guessing models. 
Using Namba forcing, we construct a model in which there are 
stationarily many indestructibly weakly guessing models which have a 
bounded countable subset not covered by any countable set in the model.
\end{abstract}

\maketitle

Weiss \cite{weiss} introduced the combinatorial principle $\textsf{ISP}(\kappa)$, 
which characterizes supercompactness in the case that $\kappa$ is inaccessible, 
but is also consistent for small values of $\kappa$ such as $\omega_2$. 
The principle $\textsf{ISP}(\omega_2)$ follows from \textsf{PFA}, and 
$\textsf{ISP}(\omega_2)$ implies some of the strong consequences of \textsf{PFA}, 
such as the failure of the square principle at all uncountable cardinals. 
Viale-Weiss \cite{vialeweiss} introduced the idea of an $\omega_1$-guessing 
model, and proved that $\textsf{ISP}(\omega_2)$ is equivalent to the existence 
of stationarily many $\omega_1$-guessing models in $P_{\omega_2}(H(\theta))$, 
for all cardinals $\theta \ge \omega_2$.

Viale \cite{viale} proved that the singular cardinal hypothesis 
(\textsf{SCH}) follows from the 
existence of stationarily many $\omega_1$-guessing models which are also 
internally unbounded, which means that any countable subset of the model is 
covered by a countable set in the model. 
This raises the question of whether $\textsf{ISP}(\omega_2)$ alone implies \textsf{SCH}. 
A closely related question of Viale \cite[Remark 4.3]{viale} is 
whether it is consistent to have 
$\omega_1$-guessing models which are not internally unbounded. 
Much of the work in this paper was motivated by these two questions.

In this paper we introduce a weak form of $\omega_1$-guessing. 
Let $\kappa$ be a regular uncountable cardinal. 
A model $N$ of size $\omega_1$ with $\kappa \in N$ 
is said to be \emph{weakly $\kappa$-guessing} 
if whenever $f : \sup(N \cap \kappa) \to On$ is a function such that for cofinally many 
$\alpha < \sup(N \cap \kappa)$, $f \restriction \alpha \in N$, 
then there is a function $g \in N$ with domain $\kappa$ such that 
$g \restriction \sup(N \cap \kappa) = f$. 
We say that $N$ is \emph{weakly guessing} if $N$ is weakly $\kappa$-guessing 
for all regular uncountable cardinals $\kappa \in N$. 
We will show that the existence of stationarily many weakly guessing 
models suffices to prove most of the strong consequences of 
$\textsf{ISP}(\omega_2)$, including the failure of square principles.

By passing through the idea of a weakly $\omega_1$-guessing model, 
we solve an easy special case of the problem of Viale stated above, showing that 
the existence of $\omega_1$-guessing models implies the existence of 
$\omega_1$-guessing models $N$ such that $\sup(N \cap On)$ has 
cofinality $\omega$, 
and in particular, $N$ is not internally unbounded. 
This result suggests a refined version of Viale's question: 
is every \emph{bounded} 
countable subset of an $\omega_1$-guessing model covered by a countable set in the model? 
If the answer is yes, then $\textsf{ISP}(\omega_2)$ does indeed imply SCH. 
This problem remains open.

The main result of this paper is the consistency that 
there are stationarily many 
indestructibly weakly guessing models $N \in P_{\omega_2}(H(\aleph_{\omega+1}))$ 
for which there is a countable subset of $N \cap \aleph_\omega$ which is not 
covered by any countable set in $N$. 
This result can be thought of as a first attempt towards proving that 
$\textsf{ISP}(\omega_2)$ does not imply \textsf{SCH}. 
The proof involves constructing a model in which there exists a diagonal form 
of Namba forcing which satisfies a weak version of the 
$\omega_1$-approximation property.

\begin{maintheorem}
	It is consistent relative to the existence of a supercompact cardinal 
	with infinitely many measurable cardinals above it that 
	there exist stationarily many $N \in P_{\omega_2}(H(\aleph_{\omega+1}))$ 
	such that $N$ is indestructibly weakly guessing, has uniform cofinality $\omega_1$, 
	and is not internally unbounded.	
	\end{maintheorem}

Before starting the main line of results concerning weakly guessing models 
and covering, we begin the paper by proving 
a variation of the classical Easton's lemma for strongly proper forcing: 
if $\p$ is strongly proper on a stationary set and $\q$ is 
$\omega_1$-closed, then $\p$ forces that $\q$ is $\omega_1$-distributive. 
As a corollary, we will show that a certain combinatorial principle known to follow 
from the existence of an indestructible version of an $\omega_1$-guessing model 
does not follow from the existence of $\omega_1$-guessing models.

\bigskip

We give a brief outline of the contents of the paper. 
Section 1 describes some background material 
which will be needed to understand the paper. 
Section 2 presents our strongly proper variation of Easton's lemma.

Section 3 defines weak approximation and weak guessing, and proves that 
many of the strong consequences of $\textsf{ISP}(\omega_2)$ 
follow from the existence of 
stationarily many weakly guessing models. 
Section 4 solves a problem of Viale \cite{viale} by showing 
the consistency that 
there are stationarily many $\omega_1$-guessing models which are not 
internally unbounded. 

Section 5 shows that the method of Viale-Weiss \cite{vialeweiss} for applying 
forcing axioms to prove the existence of $\omega_1$-guessing models 
can be adapted to the case of weakly guessing models. 
Section 6 develops a version of Namba forcing 
which has the weak approximation property. 
Section 7 proves the main theorem of the paper, 
showing that it is consistent that there exist 
stationarily many indestructibly weakly guessing models which have a 
bounded countable subset not covered by any countable set in the model.

\bigskip

We would like to thank Thomas Gilton for carefully proofreading 
several drafts of this paper and making many useful suggestions.

\section{Background}

We assume that the reader has a working knowledge of forcing, 
proper forcing, product forcing, the product lemma, 
finite step forcing iterations, and generalized stationarity. 
The reader should be familiar with trees of height $\omega_1$ and 
with the standard $\omega_1$-c.c.\ forcing poset for adding a specializing function 
to a tree with no uncountable chains. 

For a regular uncountable cardinal $\kappa$, we say that 
a forcing poset 
is \emph{$\kappa$-closed} if any descending sequence of conditions of 
length less than $\kappa$ has a lower bound. 
We say that a forcing poset is \emph{$\kappa$-distributive} if it does 
not add new sets of ordinals of size less than $\kappa$.

\begin{definition}
Let $N$ be a set with $|N| = \omega_1$ and 
$\omega_1 \subseteq N$. 
We say that $N$ has \emph{bounded uniform cofinality $\omega_1$} if for all 
$\alpha \in N$, if $N \models \cf(\alpha) > \omega$ then 
$\cf(\sup(N \cap \alpha)) = \omega_1$. 
If in addition we have that $\cf(\sup(N \cap On)) = \omega_1$, then we say that 
$N$ has \emph{uniform cofinality $\omega_1$}.
\end{definition}

Note that if $N \prec H(\theta)$ for some cardinal 
$\theta \ge \omega_2$, 
then $N$ has bounded uniform cofinality $\omega_1$ iff for every 
regular uncountable cardinal $\lambda \in N$, 
$\cf(\sup(N \cap \lambda)) = \omega_1$.

A set of ordinals $a$ is said to be \emph{countably closed} if 
every limit point of $a$ with countable cofinality is in $a$.

The following fact is well-known.

\begin{lemma}
Let $N$ be a set with $|N| = \omega_1$ and 
$\omega_1 \subseteq N$. 
Suppose that $N$ is an elementary substructure of $H(\theta)$ for 
some cardinal $\theta \ge \omega_2$ 
and $N$ has bounded uniform cofinality $\omega_1$. 
Then for any ordinal $\alpha \in N$ with uncountable cofinality, 
$N \cap \alpha$ is countably closed.
\end{lemma}

\begin{proof}
Let $\beta$ be a limit point of $N \cap \alpha$ with countable cofinality, 
and we will show that $\beta \in N \cap \alpha$. 
Since $\cf(\sup(N \cap \alpha)) = \omega_1$, $\beta < \sup(N \cap \alpha)$. 
Let $\xi := \min((N \cap \alpha) \setminus \beta)$. 

Suppose for a contradiction that $\beta \notin N$. 
Then $\beta < \xi$ and $N \cap \xi = N \cap \beta$. 
Clearly $\cf(\xi) > \omega$, for otherwise by elementarity 
$N \cap \xi$ would be cofinal in $\xi$. 
Since $N$ has bounded uniform cofinality $\omega_1$, 
$\cf(\sup(N \cap \xi)) > \omega$. 
But $\sup(N \cap \xi) = \beta$ and $\cf(\beta) = \omega$, which is 
a contradiction.
\end{proof}

\begin{definition}
Let $N$ be an uncountable set with $\omega_1 \subseteq N$. 
We say that $N$ is \emph{$\omega_1$-guessing} 
if for any set of ordinals 
$d \subseteq N$ such that $\sup(d) < \sup(N \cap On)$, 
if $d$ satisfies that for any 
countable set $b \in N$, 
$d \cap b \in N$, then there exists $d' \in N$ 
such that $d = d' \cap N$.
\end{definition}

\begin{definition}
Let $W_1$ and $W_2$ be transitive sets or classes 
with $W_1 \subseteq W_2$. 
We say that the pair $(W_1,W_2)$ has the 
\emph{$\omega_1$-approximation property} 
if whenever $d \in W_2$ is a bounded 
subset of $W_1 \cap On$ and satisfies that $b \cap d \in W_1$ 
for any set $b \in W_1$ which is countable in $W_1$, 
then $d \in W_1$.
\end{definition}

\begin{lemma}
Let $N$ be an elementary substructure of some transitive structure 
which satisfies $\textsf{ZFC} - \textrm{Powerset}$ and correctly 
computes $\omega_1$. 
Then the following are equivalent:
\begin{enumerate}
\item $N$ is $\omega_1$-guessing;
\item the pair $(N_0,V)$ has the $\omega_1$-approximation 
property, where $N_0$ is the transitive collapse of $N$.
\end{enumerate}
\end{lemma}

\begin{proof}
See the proof of \cite[Lemma 1.10]{jk26}.
\end{proof}

\begin{definition}
A forcing poset $\p$ is said to have the 
\emph{$\omega_1$-approximation property} if 
$\p$ forces that $(V,V^\p)$ has the $\omega_1$-approximation property.
\end{definition}

For a set $M$ and a filter $G$ on a forcing poset $\p$, we say that 
$G$ is \emph{$M$-generic} if for any dense subset $D$ of $\p$ which is 
a member of $M$, $G \cap M \cap D \ne \emptyset$.

\begin{definition}
Let $\p$ be a forcing poset and $N$ a set with $N \cap \p \ne \emptyset$. 
A condition $q$ is said to be \emph{strongly $(N,\p)$-generic} if 
for any dense subset $D$ of the forcing poset $N \cap \p$, 
$D$ is predense below $q$. 
\end{definition}

If $\p$ is understood from context, we say that $q$ is strongly $N$-generic 
if $q$ is strongly $(N,\p)$-generic.

Under some non-triviality assumptions on $N$, it is easy to check that 
$q$ is strongly $(N,\p)$-generic iff $q$ forces that $\dot G_\p \cap N$ 
is a $V$-generic filter on $N \cap \p$. 
For example, this is true if any conditions $s$ and $t$ in $N \cap \p$ 
which are compatible in $\p$ are also compatible in $N \cap \p$.

\begin{definition}
Let $\p$ be a forcing poset and $N$ a set. 
We say that $\p$ is \emph{strongly proper for $N$} if 
for all $p \in N \cap \p$, there is $q \le p$ such that 
$q$ is strongly $(N,\p)$-generic.
\end{definition}

\begin{definition}
Let $\kappa$ be a regular uncountable cardinal. 
A forcing poset $\p$ is \emph{$\kappa$-strongly proper on a stationary 
set} if for all sufficiently large cardinals $\chi \ge \kappa$ with 
$\p \in H(\chi)$, there are stationarily many $N \in P_{\kappa}(H(\chi))$ 
with $N \cap \kappa \in \kappa$ such that $\p$ is strongly proper for $N$. 
We say that $\p$ is \emph{$\kappa$-strongly proper} if the stationary set just 
described contains a club.
\end{definition}

A forcing poset  $\p$ is \emph{strongly proper on a stationary set} 
if $\p$ is $\omega_1$-strongly proper on a stationary set, and 
$\p$ is \emph{strongly proper} if $\p$ is $\omega_1$-strongly proper. 
The ideas of strong genericity and strongly proper are due to Mitchell.

As discussed in \cite[Section 2]{jk26}, letting 
$\lambda_\p$ be the first cardinal greater than or 
equal to $\kappa$ such that $\p \subseteq H(\lambda_\p)$, 
$\p$ is $\kappa$-strongly proper on a stationary set 
iff there are stationarily many $N \in P_{\kappa}(H(\lambda_\p))$ 
with $N \cap \kappa \in \kappa$ 
such that $\p$ is strongly proper for $N$.

The following well-known result is due to Mitchell; 
see \cite[Lemma 6]{mitchell2}.

\begin{thm}
Let $\kappa$ be a regular uncountable cardinal, and assume that 
$\p$ is $\kappa$-strongly proper on a stationary set. 
Then $\p$ forces that for any set of ordinals $X$, if 
$X \cap a \in V$ for any set $a \in V$ of size less than $\kappa$ in $V$, 
then $X \in V$. 

In particular, if $\p$ is strongly proper on a stationary set, 
then $\p$ has the $\omega_1$-approximation property.
\end{thm}

Weiss introduced principles $\textsf{ITP}(\kappa)$ and 
$\textsf{ISP}(\kappa)$ which give a combinatorial characterization 
of supercompactness in the case that $\kappa$ is inaccessible, but also 
make sense when $\kappa$ is a small cardinal such as $\omega_2$. 
We refer the reader to \cite[Section 2]{weiss} for the definitions. 
The principle $\textsf{ISP}(\kappa)$ implies $\textsf{ITP}(\kappa)$, 
and $\textsf{ISP}(\omega_2)$ follows from \textsf{PFA}.

The following theorem was proved in \cite[Section 3]{vialeweiss}.

\begin{thm}
The principle $\textsf{ISP}(\omega_2)$ is equivalent to the statement 
that for all cardinals $\chi \ge \omega_2$, there are stationarily many 
$N \in P_{\omega_2}(H(\chi))$ such that $N$ is $\omega_1$-guessing.
\end{thm}

For a cardinal $\lambda \ge \omega_2$, let $\textsf{GMP}(\lambda)$ 
be the statement that there exist stationarily many 
sets in $P_{\omega_2}(H(\lambda))$ which are $\omega_1$-guessing. 
Let $\textsf{GMP}$ be the statement that 
$\textsf{GMP}(\lambda)$ holds for all cardinals $\lambda \ge \omega_2$. 
By Theorem 1.11, $\textsf{GMP}$ is equivalent to 
$\textsf{ISP}(\omega_2)$. 

For a cardinal $\lambda \ge \omega_2$, let 
$\textsf{IGMP}(\lambda)$ be the statement that there 
exist stationarily many 
sets in $P_{\omega_2}(H(\lambda))$ which are $\omega_1$-guessing 
in any generic extension by an $\omega_1$-preserving forcing poset. 
Let $\textsf{IGMP}$ be the statement that 
$\textsf{IGMP}(\lambda)$ holds for all cardinals $\lambda \ge \omega_2$. 
We refer the reader to \cite{jk28} for more information about 
indestructible guessing models.

\section{A variation of Easton's lemma}

Before starting the main topic of the paper, we will prove in this 
section a strongly proper variation of Easton's lemma, together with 
a corollary which further distinguishes the principles 
\textsf{GMP} and \textsf{IGMP}.

By \emph{{Todor\v cevi\' c}'s maximality principle} we will 
mean the statement that any forcing poset which adds a new subset of 
$\omega_1$ whose proper initial segments are in the ground model 
collapses either $\omega_1$ or $\omega_2$. 
This statement was introduced by {Todor\v cevi\' c} \cite{todorcevic} 
and shown to follow from some combinatorial assumptions about trees.

In \cite[Theorem 3.9]{jk28} we proved that the principle 
\textsf{IGMP} together with 
$2^\omega \le \omega_2$ imply {Todor\v cevi\' c}'s maximality principle. 
We will show below that \textsf{IGMP} cannot be weakened to 
\textsf{GMP} in this result. 
The proof will use a new variation of the classical Easton's lemma.

\begin{thm}[Easton's lemma]
Suppose that $\kappa$ is a regular uncountable cardinal, $\p$ is 
$\kappa$-c.c., and $\q$ is $\kappa$-closed. 
Then $\p$ forces that $\q$ is $\kappa$-distributive.
\end{thm}

In particular, if $\p$ is $\omega_1$-c.c.\ and $\q$ is 
$\omega_1$-closed, then $\p$ forces that $\q$ is 
$\omega_1$-distributive.

We introduce a variation of Easton's lemma in which 
the same conclusion follows from $\p$ being 
strongly proper on a stationary set in place of 
being $\omega_1$-c.c. 
We will use the next result which we proved 
in \cite[Theorem 5.5]{jk28}.

\begin{thm}
Suppose that $\p$ is strongly proper on a stationary set and 
$\q$ is proper. 
Then $\q$ forces that $\p$ is strongly proper on a stationary set.
\end{thm}

\begin{thm}[Variation of Easton's lemma]
Suppose that $\p$ is strongly proper on a stationary set and 
$\q$ is $\omega_1$-closed. 
Then $\p$ forces that $\q$ is $\omega_1$-distributive.
\end{thm}

\begin{proof}
Let $G \times H$ be a generic filter on $\p \times \q$. 
By the product lemma, we have that 
$V[G \times H] = V[G][H] = V[H][G]$, where $H$ is a $V[G]$-generic 
filter on $\q$ and $G$ is a $V[H]$-generic filter on $\p$. 
So to show that $\q$ is $\omega_1$-distributive in $V[G]$, 
it suffices to show that if $f : \omega \to On$ 
is a function in $V[G \times H]$, then $f \in V[G]$. 
Fix a $\p$-name $\dot{f}$ in $V[H]$ such that 
$\dot{f}^{G} = f$.

Since $\q$ is $\omega_1$-closed, it is proper. 
As $\p$ is strongly proper on a stationary set in $V$, 
it follows by Theorem 2.2 that $\p$ is 
strongly proper on a stationary set in $V[H]$.

Working in $V[H]$, 
fix a regular cardinal $\theta$ large enough so that 
$\p$, $\q$, and $\dot f$ are in $H(\theta)$. 
Since $\p$ is strongly proper on a stationary set in $V[H]$, it follows that 
there are stationarily many 
$M \in P_{\omega_1}(H(\theta))$ 
with $M \prec (H(\theta),\in,\p,\dot f)$ such that 
every condition 
in $M \cap \p$ has an extension which is strongly $(M,\p)$-generic.

Applying the $V[H]$-genericity of $G$, 
an easy density argument shows that 
for some $M$ as described in the previous paragraph, 
$G$ contains a strongly $(M,\p)$-generic condition. 
Since $G$ contains a strongly $(M,\p)$-generic condition, 
it follows that $g_M := G \cap M$ is a $V[H]$-generic filter 
on the forcing poset $M \cap \p$. 

Since $\dot f \in M$, by elementarity we have that for each $n < \omega$, 
the dense set $D_n$ 
of conditions in $\p$ which decide the value of $\dot f(n)$ is in $M$. 
Again by elementarity, it follows that $D_n \cap M$ is a dense subset of 
$M \cap \p$. 
Since $g_M$ is a $V[H]$-generic filter on $M \cap \p$, 
$g_M \cap D_n \ne \emptyset$. 
It easily follows that for all $n < \omega$ and $\alpha$, 
$$
f(n) = \alpha \iff 
\exists t \in g_M,\ t \Vdash^{V[H]}_\p \dot f(n) = \check \alpha.
$$
So $f$ is definable in the model $V[H][g_M]$, and hence is in 
$V[H][g_M]$.

The forcing poset $M \cap \p$ is countable. 
Since $\q$ is $\omega_1$-closed and $\p \in V$, 
it follows that $M \cap \p \in V$. 
As $M \cap \p$ and $\q$ are both in $V$, 
the product lemma implies that $V[H][g_M] = V[g_M][H]$. 
As $M \cap \p$ is countable, it is $\omega_1$-c.c. 
So by Easton's lemma, $\q$ is $\omega_1$-distributive in $V[g_M]$. 
But $f \in V[H][g_M] = V[g_M][H]$ and $f$ is countable, 
so $f \in V[g_M]$. 
And $V[g_M] = V[G \cap M] \subseteq V[G]$, so $f \in V[G]$.
\end{proof}

\begin{corollary}
The principle $\textsf{GMP}$ together with $2^\omega \le \omega_2$ 
is consistent with the existence of an $\omega_1$-distributive 
nowhere c.c.c.\! forcing poset of size $\omega_1$. 
In particular, $\textsf{GMP}$ together with $2^\omega \le \omega_2$ does not 
imply {Todor\v cevi\' c}'s maximality principle.
\end{corollary}

\begin{proof}
Let $\kappa$ be a supercompact cardinal. 
Then there is a strongly proper forcing poset $\p$ which collapses $\kappa$ 
to become $\omega_2$, forces that $2^\omega = \omega_2$, 
and forces $\textsf{GMP}$. 
For example, let $\p$ be the forcing poset consisting of 
finite adequate sets of countable models, ordered by reverse inclusion. 
See Sections 6 and 7 of \cite{jk26} for the details.

Let $G$ be a generic filter on $\p$. 
Let $\q := \add(\omega_1)^V$. 
Then $\q$ is a non-trivial forcing poset in $V[G]$. 
In $V$, $|\q| = 2^\omega < \kappa$, so in $V[G]$, 
$|\q| < \kappa = \omega_2$. 
Hence, $|\q| = \omega_1$ in $V[G]$. 
In $V$, $\p$ is strongly proper on a stationary set and 
$\q$ is $\omega_1$-closed. 
By Theorem 2.3, it follows that $\q$ is $\omega_1$-distributive in $V[G]$.
\end{proof}

We comment that in the model of \cite[Section 7]{jk26}, \textsf{GMP} holds but there exists 
an $\omega_1$-Suslin tree. 
Since an $\omega_1$-Suslin tree is an example of a c.c.c., $\omega_1$-distributive 
forcing of size $\omega_1$, this provides a different proof that {Todor\v cevi\' c}'s 
maximality principle does not follow from \textsf{GMP}.

Let us give another example of a strongly proper 
variation of a classical result.

\begin{thm}
Let $2^\omega = \omega_2$ and suppose that $T$ is an 
$\omega_2$-Aronszajn tree. 
Assume that $\q$ is an $\omega_1$-closed forcing poset which collapses 
$\omega_2$, and $\p$ is a forcing poset which is $\omega_1$-Knaster 
in $V^\q$. 
Then $\p \times \q$ does not add a cofinal branch of $T$.
\end{thm}

\begin{proof}
See \cite[Lemma 23.1]{cummings}.
\end{proof}

\begin{thm}
Suppose that $T$ is a tree whose height is an ordinal with 
uncountable cofinality, all of whose levels have size 
less than $2^\omega$. 
Let $\p$ be strongly proper on a stationary set and $\q$ be
$\omega_1$-closed. 
Then $\p \times \q$ does not add any new cofinal branches to $T$.
\end{thm}

\begin{proof}
Let $G \times H$ be a $V$-generic filter on $\p \times \q$. 
By the product lemma, $V[G \times H] = V[H][G]$. 
By the proof of \cite[Lemma 23.1]{cummings}, 
$\omega_1$-closed 
forcing cannot add new cofinal branches to $T$, so 
there are no new cofinal branches of $T$ in $V[H]$. 
By Theorem 2.2, $\p$ is still strongly proper on a stationary set 
in $V[H]$, and hence has the $\omega_1$-approximation property in $V[H]$. 
Hence, $\p$ cannot add any new cofinal branches of $T$ over $V[H]$. 
Therefore, $V[H][G] = V[G \times H]$ has no new cofinal branches of $T$.
\end{proof}

\bigskip

Theorems 2.3 and 2.6 can be generalized to higher cardinals. 
First, the following generalization of Theorem 2.2 
follows by a straightforward modification of the original argument 
given in \cite[Theorem 5.5]{jk28}. 

\begin{thm}
Let $\kappa$ be a regular uncountable cardinal. 
Let $\p$ and $\q$ be forcing posets, 
where $\q$ is $\kappa$-closed. 
Suppose that for all large enough regular cardinals $\theta$, 
there are stationarily many internally approachable 
$N \in P_{\kappa}(H(\theta))$ with $N \cap \kappa \in \kappa$ 
such that $\p$ is strongly proper for $N$. 
Then $\q$ forces that $\p$ is $\kappa$-strongly proper on a stationary set.
\end{thm}

We need to work with internally approachable models $N$ in 
$P_{\kappa}(H(\theta))$ because those 
are the models for which $\kappa$-closed forcings have 
$N$-generic conditions.

Secondly, the proofs of Theorems 2.3 and 2.6 can be modified in the 
obvious way to prove the following results. 
We leave the details for the interested reader.

\begin{thm}
Let $\kappa$ be a regular uncountable cardinal. 
Let $\p$ and $\q$ be forcing posets, where $\q$ is 
$\kappa$-closed. 
Suppose that for all large enough regular cardinals $\theta$, 
there are stationarily many internally approachable models 
$N \in P_{\kappa}(H(\theta))$ with $N \cap \kappa \in \kappa$ 
such that $\p$ is strongly proper for $N$. 
Then $\p$ forces that $\q$ is $\kappa$-distributive.
\end{thm}

\begin{thm}
Let $\mu < \kappa$ be regular cardinals with $2^{<\mu} < \kappa$. 
Let $T$ be a tree whose height is an ordinal with 
cofinality at least $\kappa$ such that 
the levels of $T$ all have size less than $2^\mu$. 
Suppose that $\p$ and $\q$ are forcing posets, where $\q$ is 
$\kappa$-closed, and assume that 
for all large enough regular cardinals $\theta$, 
there are stationarily many internally approachable models 
$N \in P_{\kappa}(H(\theta))$ with $N \cap \kappa \in \kappa$ 
such that $\p$ is strongly proper for $N$. 
Then $\p \times \q$ does not add new cofinal branches to $T$.
\end{thm}

\section{Weak approximation and guessing}

We now begin the main topic of the paper by introducing a 
weak form of guessing. 
We will show that the existence of stationarily 
many weakly guessing 
models implies many of the same strong consequences as 
the existence of stationarily many $\omega_1$-guessing models.

For the remainder of the paper, we will say that $N$ is an 
\emph{elementary substructure} to mean that $N$ is an elementary 
substructure of some transitive set which models 
{\textsf{ZFC} - \textrm{Powerset}} and correctly computes $\omega_1$.

\begin{definition}
Let $N$ be set of size $\omega_1$ with $\omega_1 \subseteq N$.

Let $\kappa \in N$ be an ordinal such that $N$ models that 
$\kappa$ is regular uncountable. 
We say that $N$ is 
\emph{weakly $\kappa$-guessing} 
if whenever $f : \sup(N \cap \kappa) \to On$ is a function 
such that for cofinally many $\alpha < \sup(N \cap \kappa)$, 
$f \restriction \alpha \in N$, then there is a function 
$g \in N$ with domain $\kappa$ such that $g \restriction \sup(N \cap \kappa) = f$. 

We say that $N$ is \emph{weakly guessing} if for any $\kappa \in N$ which 
$N$ models is regular uncountable, 
$N$ is weakly $\kappa$-guessing.
\end{definition}

Note that in the case that $\kappa = \omega_1$, $N$ is weakly $\omega_1$-guessing 
iff whenever $f : \omega_1 \to On$ is a function such that for all 
$\alpha < \omega_1$, $f \restriction \alpha \in N$, then $f \in N$.

\begin{lemma}
Let $N$ be an elementary substructure of size $\omega_1$ with 
$\omega_1 \subseteq N$. 
Let $\kappa \in N$ be a regular uncountable cardinal, and assume 
that $N$ is weakly $\kappa$-guessing. 
Then $\cf(\sup(N \cap \kappa)) = \omega_1$.
\end{lemma}

\begin{proof}
Suppose for a contradiction that $b$ is a cofinal subset of $N \cap \kappa$ 
with order type $\omega$. 
Let $f : \sup(N \cap \kappa) \to 2$ be the function such that 
$f(\gamma) = 1$ iff $\gamma \in b$. 
Then easily for all $\alpha \in N \cap \kappa$, $f \restriction \alpha \in N$, 
since $f \restriction \alpha$ is the characteristic function of the finite 
set $b \cap \alpha$.

Since $N$ is weakly $\kappa$-guessing, there is a function 
$g : \kappa \to On$ in $N$ such that $g \restriction \sup(N \cap \kappa) = f$. 
For all $\alpha \in N \cap \kappa$, 
$g \restriction \alpha = f \restriction \alpha$ is the characteristic 
function of a finite set. 
So by the elementarity of $N$, for all $\alpha \in \kappa$, 
$g \restriction \alpha$ is the characteristic function of a finite set. 
In particular, $g \restriction \sup(N \cap \kappa) = f$ is the 
characteristic function of a finite set, which contradicts the 
fact that $b$ is infinite.
\end{proof}

As with the property of being $\omega_1$-guessing, we can characterize 
the property of a set being weakly guessing 
in terms of an approximation property 
of its transitive collapse.

\begin{definition}
Let $W_1$ and $W_2$ be transitive sets or classes 
with $W_1 \subseteq W_2$, and let $\lambda \in W_1$ be an ordinal. 
We say that the pair $(W_1,W_2)$ has the 
\emph{weak $\lambda$-approximation property} if whenever 
$f : \lambda \to On$ is a function in $W_2$ such that for all 
$\alpha < \lambda$, $f \restriction \alpha \in W_1$, then $f \in W_1$.
\end{definition}

\begin{lemma}
If $(W_1,W_2)$ has the $\omega_1$-approximation property, 
where $W_1$ is closed under intersections, 
then $(W_1,W_2)$ has the weak $\lambda$-approximation property for 
all ordinals $\lambda \in W_1$ with uncountable cofinality in $W_1$.
\end{lemma}

\begin{proof}
Let $f : \lambda \to On$ be a function in $W_2$, 
where $W_1 \models \cf(\lambda) \ge \omega_1$, and 
suppose that for all $\alpha < \lambda$, $f \restriction \alpha \in W_1$. 
We claim that $f$ is in $W_1$. 
Since $(W_1,W_2)$ has the $\omega_1$-approximation property, it suffices 
to show that if $a$ is countable in $W_1$, then $a \cap f \in W_1$. 
Since $\lambda$ has uncountable cofinality in $W_1$, 
there is $\beta < \lambda$ such that $\dom(a) \cap \lambda \subseteq \beta$. 
Then $a \cap f = a \cap (f \restriction \beta)$. 
Since $a$ and $f \restriction \beta$ are in $W_1$, so is $a \cap f$.
\end{proof}

\begin{lemma}
Let $N$ be an elementary substructure of size $\omega_1$ such that 
$\omega_1 \subseteq N$. 
Let $\kappa \in N$ be a regular uncountable cardinal. 
Then $N$ is weakly $\kappa$-guessing iff 
the pair $(N_0,V)$ has the weak $\pi(\kappa)$-approximation property, 
where $\pi : N \to N_0$ is the transitive collapse of $N$.
\end{lemma}

\begin{proof}
Assume that $N$ is weakly $\kappa$-guessing, and we will prove 
that $(N_0,V)$ has the weak $\pi(\kappa)$-approximation property. 
By Lemma 3.2, $\cf(\sup(N \cap \kappa)) = \omega_1$. 
Let $f : \pi(\kappa) \to On$ be a function such that for all 
$\alpha < \pi(\kappa)$, $f \restriction \alpha \in N_0$. 
We will show that $f \in N_0$.

Define 
$$
f^* := \bigcup \{ \pi^{-1}(f \restriction \alpha) : \alpha < \pi(\kappa) \}.
$$
It is easy to check that $f^*$ is an ordinal-valued function with 
domain $\sup(N \cap \kappa)$. 
For all $\beta \in N \cap \kappa$, 
$f^* \restriction \beta = \pi^{-1}(f \restriction \pi(\beta))$, so for 
cofinally many $\alpha < \sup(N \cap \kappa)$, $f^* \restriction \alpha \in N$. 
Since $N$ is weakly $\kappa$-guessing, there is 
$g : \kappa \to On$ in $N$ such that 
$g \restriction \sup(N \cap \kappa) = f^*$. 
It is easy to check that $\pi(g) = f$, and hence $f \in N_0$.

Conversely, assume that $(N_0,V)$ has the weak 
$\pi(\kappa)$-approximation property, and we will prove that $N$ 
is weakly $\kappa$-guessing. 
Assume that $f : \sup(N \cap \kappa) \to On$ is a function such that for 
cofinally many $\alpha < \sup(N \cap \kappa)$, $f \restriction \alpha \in N$. 
Note that for all $\beta \in N \cap \kappa = N \cap \sup(N \cap \kappa)$ , 
$f \restriction \beta \in N$.

Let 
$$
f^* := \bigcup \{ \pi(f \restriction \alpha) : \alpha \in N \cap \kappa \}.
$$
It is easy to check that $f^*$ is an ordinal-valued function with domain 
$\pi(\kappa)$, and for all 
$\beta < \pi(\kappa)$, 
$f^* \restriction \beta = \pi(f \restriction \pi^{-1}(\beta)) \in N_0$. 
Since $(N_0,V)$ has the weak $\pi(\kappa)$-approximation property, 
$f^* \in N_0$. 
Then $\pi^{-1}(f^*) \in N$, and it is easy to check that 
$\pi^{-1}(f^*) \restriction \sup(N \cap \kappa) = f$.
\end{proof}

\begin{corollary}
Let $N$ be an elementary substructure of size $\omega_1$ such that 
$\omega_1 \subseteq N$. 
If $N$ is $\omega_1$-guessing, then $N$ is weakly 
$\kappa$-guessing for all regular uncountable cardinals $\kappa \in N$.
\end{corollary}

\begin{proof}
Let $\pi : N \to N_0$ be the transitive collapse of $N$. 
Suppose that $N$ is $\omega_1$-guessing. 
Then by Lemma 1.5, $(N_0,V)$ has the $\omega_1$-approximation 
property. 
Let $\kappa \in N$ be regular uncountable. 
Then $N_0$ models that $\pi(\kappa)$ is regular uncountable, and 
hence has uncountable cofinality. 
So $(N_0,V)$ has the weak $\pi(\kappa)$-approximation property 
by Lemma 3.4. 
By Lemma 3.5, $N$ is weakly $\kappa$-guessing.
\end{proof}

\begin{definition}
Let $\kappa$ be a regular uncountable cardinal. 
A forcing poset $\p$ is said to have the 
\emph{weak $\kappa$-approximation property} if 
$\p$ forces that the pair $(V,V^\p)$ has the weak 
$\kappa$-approximation property.
\end{definition}

\begin{lemma}
Let $\kappa$ be a regular uncountable cardinal. 
If $\p$ has the weak $\kappa$-approximation property, then 
$\p$ forces that $\cf(\kappa) > \omega$. 
In particular, if $\p$ has the weak $\omega_1$-approximation property, 
then $\p$ preserves $\omega_1$.
\end{lemma}

\begin{proof}
Similar to the proof of Lemma 3.2.
\end{proof}

We can also define an indestructible version of a weakly guessing model.

\begin{definition}
Let $N$ be a set of size $\omega_1$ with $\omega_1 \subseteq N$. 

Let $\kappa \in N$ be a regular uncountable cardinal. 
We say that $N$ is \emph{indestructibly weakly $\kappa$-guessing} 
if $N$ is weakly $\kappa$-guessing in any generic extension by an 
$\omega_1$-preserving forcing poset.

We say that $N$ is \emph{indestructibly weakly guessing} if for all regular uncountable 
cardinals $\kappa \in N$, $N$ is indestructibly weakly $\kappa$-guessing.
\end{definition}

\bigskip

With the idea of a weakly guessing model at hand, we 
can now introduce principles which weaken $\textsf{GMP}$ and 
$\textsf{IGMP}$.

\begin{definition}
For a cardinal $\lambda \ge \omega_2$, 
let $\textsf{wGMP}(\lambda)$ be the statement that 
there exist stationarily many sets $N \in P_{\omega_2}(H(\lambda))$ such that 
$\omega_1 \subseteq N$ and 
$N$ is weakly guessing. 
Let $\textsf{wGMP}$ be the statement that 
$\textsf{wGMP}(\lambda)$ holds for all cardinals $\lambda \ge \omega_2$.
\end{definition}

\begin{definition}
For a cardinal $\lambda \ge \omega_2$, 
let $\textsf{wIGMP}(\lambda)$ be the statement that 
there exist 
stationarily many sets $N \in P_{\omega_2}(H(\lambda))$ such that 
$\omega_1 \subseteq N$ and 
$N$ is indestructibly weakly guessing. 
Let $\textsf{wIGMP}$ be the statement that 
$\textsf{wIGMP}(\lambda)$ holds for all cardinals $\lambda \ge \omega_2$.
\end{definition}

\begin{corollary}
Let $\lambda \ge \omega_2$. 
Then $\textsf{GMP}(\lambda)$ implies $\textsf{wGMP}(\lambda)$, and 
$\textsf{IGMP}(\lambda)$ implies $\textsf{wIGMP}(\lambda)$. 
Thus, $\textsf{GMP}$ implies $\textsf{wGMP}$, and 
$\textsf{IGMP}$ implies $\textsf{wIGMP}$. 
\end{corollary}

\begin{proof}
Immediate from Corollary 3.6.
\end{proof}

\bigskip

We now explore strong consequences of the existence of weakly 
guessing models. 
The main point is that almost all of the known consequences of the 
existence of $\omega_1$-guessing models follow from
the existence of weakly guessing models.\footnote{A 
possible exception 
is the result of Viale \cite[Section 7.2]{viale} 
that the existence of stationarily many $\omega_1$-guessing 
models which are internally unbounded implies \textsf{SCH}. 
We do not know whether \textsf{SCH} 
follows from stationarily many internally unbounded 
weakly guessing models.}

It is straightforward to check that the following three 
consequences of the existence of $\omega_1$-guessing models 
follow from the existence of weakly $\omega_1$-guessing models, 
by slight modifications 
of the proofs given in \cite{jk28}.

\begin{proposition}
Suppose that there are stationarily many 
$N$ in $P_{\omega_2}(H(\omega_3))$ such that $N$ is weakly  
$\omega_1$-guessing and 
$N \cap \omega_2$ has cofinality $\omega_1$. 
Then $\neg \textsf{AP}_{\omega_1}$.
\end{proposition}

\begin{proof}
The argument of \cite[Proposition 2.6]{jk28} shows that there exists a 
function $f : \omega_1 \to N$ which is cofinal in $N \cap \omega_2$ all of 
whose initial segments are in $N$. 
Since $N$ is weakly $\omega_1$-guessing, $f \in N$. 
By elementarity, $\sup(\ran(f)) = N \cap \omega_2 \in N$, which is impossible.
\end{proof}

\begin{proposition}
Suppose that 
there are stationarily many $N$ in $P_{\omega_2}(H(\omega_2))$ 
which are weakly $\omega_1$-guessing. 
Then there does not exist a weak $\omega_1$-Kurepa tree. 
In particular, \textsf{CH} fails.
\end{proposition}

\begin{proof}
See \cite[Proposition 2.8]{jk28}.
\end{proof}

\begin{proposition}
Assume that $2^\omega \le \omega_2$ and there are cofinally many 
sets $N$ in $P_{\omega_2}(H(\omega_2))$ which are 
indestructibly weakly $\omega_1$-guessing. 
Then {Todor\v cevi\' c}'s maximality principle holds.
\end{proposition}

\begin{proof}
See \cite[Theorem 3.9]{jk28}.
\end{proof}

Weiss \cite[Section 2]{weiss} proved that the principle $\textsf{ITP}(\omega_2)$ 
implies the non-existence of $\omega_2$-Aronszajn trees.

\begin{proposition}
Suppose that there exist stationarily many weakly 
$\omega_2$-guessing models in $P_{\omega_2}(H(\omega_3))$. 
Then there does not exist an $\omega_2$-Aronszajn tree.
\end{proposition}

\begin{proof}
Let $T$ be a tree of height $\omega_2$, all of whose levels have 
cardinality less than $\omega_2$. 
We will prove that there is a branch of $T$ with order type $\omega_2$. 
Without loss of generality, assume that $T$ has underlying set $\omega_2$. 
Then $T \in H(\omega_3)$. 
Fix $N$ in $P_{\omega_2}(H(\omega_3))$ such that 
$N \prec H(\omega_3)$, $T \in N$, 
and $N$ is weakly $\omega_2$-guessing. 
Let $\alpha := N \cap \omega_2$. 
By the elementarity of $N$, any node in $T$ of height less than $\alpha$ 
is in $N$.

Fix a node $y$ on level $\alpha$ of $T$. 
Define a function $f : \alpha \to \omega_2$ by letting 
$f(\beta)$ be the member of the set $\{ x \in T : x <_T y \}$ 
which is on level $\beta$ of $T$. 
Then by elementarity, 
for all $\gamma < \alpha$, $f \restriction \gamma$ is in $N$, since 
this function just enumerates in $<_T$-increasing order the elements of 
$T$ which are $<_T$-below $f(\gamma)$. 
Since $N$ is weakly $\omega_2$-guessing, there is a function 
$g : \omega_2 \to On$ in $N$ such that $g \restriction \alpha = f$.

We claim that the range of $g$ is a cofinal branch of $T$ of order 
type $\omega_2$, which 
completes the proof. 
For all $\gamma < \beta$ in $N \cap \omega_2 = \alpha$, 
$g(\gamma) <_T g(\beta)$, where $g(\gamma)$ is on level $\gamma$ 
of $T$ and $g(\beta)$ is on level $\beta$ of $T$. 
By elementarity, the same statement holds for all $\gamma < \beta$ 
in $\omega_2$. 
It follows that the range of $g$ is a chain of $T$ which meets each 
level of $T$, and thus is a cofinal branch of $T$ 
of order type $\omega_2$.
\end{proof}

The square principle described in the next proposition was originally 
shown to fail under $\textsf{ITP}(\omega_2)$ by Weiss \cite[Section 4]{weiss}.

\begin{proposition}
Let $\kappa \ge \omega_2$ be regular, $\theta > \kappa$, and assume that 
there are stationarily 
many $N \in P_{\omega_2}(H(\theta))$ with bounded uniform cofinality 
$\omega_1$ which are weakly $\kappa$-guessing. 
Then there does not exist a sequence 
$$
\langle c_\alpha : \alpha \in \kappa \cap \cof(\omega_1) \rangle
$$
satisfying:
\begin{enumerate}
\item $c_\alpha$ is a club subset of $\alpha$;
\item for any ordinal $\gamma < \kappa$, the set 
$$
E_\gamma := \{ c_\alpha \cap \gamma : 
\alpha \in \kappa \cap \cof(\omega_1), \ 
\gamma \in \lim(c_\alpha) \}
$$
has size less than $\omega_2$;
\item there is no club set $D \subseteq \kappa$ such that 
for every $\gamma \in \lim(D) \cap \kappa$, 
$D \cap \gamma \in E_\gamma$.
\end{enumerate}
\end{proposition}

\begin{proof}
Suppose for a contradiction that 
$$
\vec c = \langle c_\alpha : \alpha \in \kappa \cap \cof(\omega_1) \rangle
$$
is such a sequence. 
Fix $N \in P_{\omega_2}(H(\theta))$ such that 
$N \prec (H(\theta),\in,\kappa,\vec c)$ and $N$ is weakly $\kappa$-guessing. 
Note that for all $\gamma \in N \cap \kappa$, 
$E_\gamma \in N$, and so $E_\gamma \subseteq N$ since 
$|E_\gamma| < \omega_2$ and $\omega_1 \subseteq N$.

Let $\alpha := \sup(N \cap \kappa)$. 
Since $\kappa$ is regular, $\alpha < \kappa$. 
By Lemma 3.2, $\cf(\alpha) = \omega_1$. 
Moreover, by Lemma 1.2, $N \cap \kappa = N \cap \alpha$ 
is countably closed. 
Since $c_\alpha$ is a club subset of $\alpha$, it follows 
that $\lim(c_\alpha) \cap N \cap \alpha$ is cofinal in $\alpha$.

Let $f : \alpha \to 2$ be the characteristic function of $c_\alpha$. 
We claim that there are cofinally many $\gamma < \alpha$ 
such that $f \restriction \gamma$ is in $N$. 
The set $\lim(c_\alpha) \cap N \cap \alpha$ is cofinal in $\alpha$. 
Consider $\gamma$ in this set. 
Then $c_\alpha \cap \gamma \in E_\gamma$. 
Since $E_\gamma \subseteq N$ as observed above, 
$c_\alpha \cap \gamma \in N$. 
But $f \restriction \gamma$ is just the characteristic function 
of $c_\alpha \cap \gamma$, and hence is in $N$.

Since $N$ is weakly $\kappa$-guessing, there is 
a function $g : \kappa \to On$ in $N$ such that 
$g \restriction \alpha = f$. 
Let $D := \{ \gamma < \kappa : g(\gamma) = 1 \}$. 
Then $D \cap \alpha = c_\alpha$. 
An easy argument using elementarity and the fact that 
$c_\alpha$ is a club in $\alpha$ shows that 
$D$ is a club subset of $\kappa$.

If $\gamma$ is a limit point of $D$ in $N \cap \kappa = N \cap \alpha$, 
then $\gamma$ is a limit point of $c_\alpha$, and hence 
$c_\alpha \cap \gamma = D \cap \gamma$ is in 
$E_\gamma$. 
By elementarity, it follows that if $\gamma$ is a limit point of $D$ in 
$\kappa$, then $D \cap \gamma \in E_\gamma$. 
This contradicts property 3 of the sequence $\vec c$.
\end{proof}

\begin{corollary}
The principle $\textsf{wGMP}$ implies $\neg \Box_\kappa$ and 
$\neg \Box(\lambda)$ for all cardinals $\kappa \ge \omega_1$ and 
regular cardinals $\lambda \ge \omega_2$.
\end{corollary}

\section{Guessing models and covering}

In this section we will prove 
that $\textsf{GMP}(\omega_2)$ implies the existence of stationarily many 
$\omega_1$-guessing models in $P_{\omega_2}(H(\omega_2))$ 
which are not internally unbounded.

\begin{lemma}
Suppose that $\langle M_n : n < \omega \rangle$ is a $\subseteq$-increasing 
sequence of sets which are weakly $\omega_1$-guessing. 
Then $M := \bigcup \{ M_n : n < \omega \}$ is weakly 
$\omega_1$-guessing.
\end{lemma}

\begin{proof}
Let $f : \omega_1 \to On$ be a function such that for 
cofinally many $\alpha < \omega_1$, $f \restriction \alpha \in M$. 
Note that since $\omega_1 \subseteq M$, this implies that for all 
$\alpha < \omega_1$, $f \restriction \alpha \in M$. 
We will prove that $f \in M$.

As $M = \bigcup \{ M_n : n < \omega \}$, 
for each $\alpha < \omega_1$ we can fix 
$n_\alpha < \omega$ such that $f \restriction \alpha \in M_{n_\alpha}$. 
Fix an uncountable set $X \subseteq \omega_1$ and $n < \omega$ such that 
for all $\alpha \in X$, $n_\alpha = n$. 

Since $X$ is uncountable, it is unbounded in $\omega_1$. 
Therefore, $\bigcup \{ f \restriction \alpha : \alpha \in X \} = f$. 
For each $\alpha \in X$, $n_\alpha = n$ implies that $f \restriction \alpha \in M_n$. 
So there are cofinally many $\alpha < \omega_1$ such that 
$f \restriction \alpha \in M_n$. 
Since $M_n$ is weakly $\omega_1$-guessing, $f \in M_n$. 
As $M_n \subseteq M$, $f \in M$, and we are done.
\end{proof}

\begin{corollary}
Let $\theta \ge \omega_2$ be a regular cardinal. 
Assume that there are stationarily many 
$M \in P_{\omega_2}(H(\theta))$ such 
that $M$ is weakly $\omega_1$-guessing. 
Then there are stationarily many $M \in P_{\omega_2}(H(\theta))$ such that 
$M$ is weakly $\omega_1$-guessing and for all regular cardinals 
$\lambda \ge \omega_2$ in $M \cup \{ \theta \}$, 
$\sup(M \cap \lambda) = \omega$.
\end{corollary}

\begin{proof}
Let $F : H(\theta)^{<\omega} \to H(\theta)$ be a function. 
Since there are stationarily many $N \in P_{\omega_2}(H(\theta))$ 
which are weakly $\omega_1$-guessing, we can inductively define a sequence 
$\langle M_n : n < \omega \rangle$ such that for all $n < \omega$, 
$M_n \prec H(\theta)$, $M_n$ is closed under $F$, 
$M_n$ is weakly $\omega_1$-guessing, and 
$M_n \in M_{n+1}$.

Let $M := \bigcup \{ M_n : n < \omega \}$. 
By Lemma 4.1, $M$ is weakly $\omega_1$-guessing. 
For any regular cardinal $\lambda$ in $M \cup \{ \theta \}$,  
the countable set 
$\{ \sup(M_n \cap \lambda) : n < \omega \}$ is cofinal in 
$\sup(M \cap \lambda)$, and hence 
$\cf(\sup(M \cap \lambda)) = \omega$.
\end{proof}

In particular, being a weakly $\omega_1$-guessing model does not 
imply being internally unbounded.

\begin{lemma}
Let $M \in P_{\omega_2}(H(\omega_2))$ be a set such that 
$\omega_1 \subseteq M$ and $M \prec H(\omega_2)$. 
Then $M$ is $\omega_1$-guessing iff $M$ is weakly $\omega_1$-guessing.
\end{lemma}

\begin{proof}
By Corollary 3.6, if $M$ is $\omega_1$-guessing, then $M$ 
is weakly $\omega_1$-guessing. 
Assume that $M$ is weakly $\omega_1$-guessing, and we will show 
that $M$ is $\omega_1$-guessing.

Let $X$ be a bounded subset of $M \cap On = M \cap \omega_2$ 
such that for any countable set $a \in M$, $a \cap X \in M$. 
We will show that $X \in M$. 
If $X$ is a bounded subset of $\omega_1$, then we can fix a countable 
ordinal $\beta$ with $\sup(X) < \beta$. 
Then by assumption, $X = X \cap \beta$ is in $M$. 
So assume that $\sup(X) \ge \omega_1$.

Since $\sup(X) < \sup(M \cap \omega_2)$, 
fix $\alpha \in M \cap \omega_2$ such that $\sup(X) < \alpha$. 
By elementarity, fix a bijection $f : \omega_1 \to \alpha$ in $M$. 
Then $Y := f^{-1}(X)$ is a subset of $\omega_1$. 
Let $g : \omega_1 \to 2$ be the characteristic function of $Y$. 
We claim that for all $\beta < \omega_1$, $g \restriction \beta \in M$. 
If $\beta < \omega_1$, then $f[\beta]$ is a countable set in $M$, and therefore 
$f[\beta] \cap X$ is in $M$. 
Hence, $f^{-1}(f[\beta] \cap X) = Y \cap \beta$ is in $M$. 
Since $g \restriction \beta$ is the characteristic function of $Y \cap \beta$, 
$g \restriction \beta \in M$.

As $M$ is weakly $\omega_1$-guessing, it follows that 
$g$ is in $M$. 
As $g$ is the characteristic function of $Y$, $Y \in M$. 
By elementarity, $f[Y] = X$ is in $M$.
\end{proof}

\begin{corollary}
The principles $\textsf{GMP}(\omega_2)$ and 
$\textsf{wGMP}(\omega_2)$ are equivalent.
\end{corollary}

\begin{proof}
Immediate from Lemma 4.3.
\end{proof}

\begin{corollary}
Assuming $\textsf{GMP}(\omega_2)$,  
there are stationarily many $M \in P_{\omega_2}(H(\omega_2))$ 
such that $\sup(M \cap \omega_2)$ has cofinality $\omega$ and 
$M$ is $\omega_1$-guessing (and in particular, $M$ is not 
internally unbounded).
\end{corollary}

\begin{proof}
Immediate from Corollary 4.2 and Lemma 4.3.
\end{proof}

\section{Forcing axioms and weak guessing}

The goal of the rest of the paper is to prove the consistency of 
the existence of stationarily 
many indestructibly weakly guessing models $N$ which have a 
countable bounded subset which is not covered by any countable set in $N$. 
We will use the following theorem of Woodin.

\begin{thm}[Woodin \cite{woodin}]
Let $\p$ be a forcing poset, and assume that for any family of $\omega_1$ many 
dense subsets of $\p$, there exists a filter on $\p$ which meets each 
set in the family. 
Then for any regular cardinal $\theta$ with $\p \in H(\theta)$, there are 
stationarily many $N \in P_{\omega_2}(H(\theta))$ with 
$\omega_1 \subseteq N$ for which there exists an $N$-generic filter 
on $\p$.
\end{thm}

Recall that a filter $G$ on $\p$ is $N$-generic if for any dense set 
$D \in N$, $D \cap G \cap N \ne \emptyset$.

\begin{lemma}
Let $\p$ be a forcing poset, 
$\theta \ge \omega_2$ a regular cardinal with $\p \in H(\theta)$, 
and $N \in P_{\omega_2}(H(\theta))$ such that $\omega_1 \subseteq N$ 
and $N \prec (H(\theta),\in,\p)$. 
Suppose that $G$ is an $N$-generic filter on $\p$. 
Assume that $p \in G$, $\lambda \in N$, 
and $p$ forces that there is a countable subset of $\lambda$ which is not 
covered by any countable set in $V$. 
Then there is a countable subset of $N \cap \lambda$ which is not covered 
by any countable set in $N$.
\end{lemma}

\begin{proof}
Suppose that $p$ forces that $\dot a$ is a countable subset of $\lambda$ 
which is not covered by any countable set in $V$. 
By elementarity and the fact that $G$ is an $N$-generic 
filter, without loss of generality we may assume that $p$ and $\dot a$ are 
in $N$. 
By elementarity, we can fix a $\p$-name $\dot f$ in $N$ which $p$ forces is 
a surjection of $\omega$ onto $\dot a$. 
Now let $b$ be the set of $\beta \in N$ such that for some $q \in G \cap N$ and 
$n < \omega$, $q \Vdash \dot f(n) = \check \beta$. 
Since $G$ is a filter, each $n$ has a unique such ordinal $\beta$, and hence 
$b$ is a countable subset of $N \cap \lambda$.

Let $c$ be a countable set in $N$, and we will show that 
$b$ is not a subset of $c$.  
Note that $p$ forces that $\dot a$ is not a subset of $c$. 
Let $D$ be the dense set of conditions which are either incompatible 
with $p$, or below $p$ and decide for some $n$ and $\beta$ 
that $\dot f(n) = \beta \in \dot a \setminus c$. 
By elementarity, $D \in N$. 
Since $G$ is $N$-generic, fix $q \in D \cap G \cap N$. 
Then $q \le p$, and for some $n$ and $\beta$ in $N$, 
$q \Vdash \dot f(n) =\beta \in \dot a \setminus c$. 
Then $\beta \in b \setminus c$, and we are done.
\end{proof}

\begin{lemma}
Let $\p$ be a forcing poset, 
$\theta \ge \omega_2$ a regular cardinal with $\p \in H(\theta)$, 
and $N \in P_{\omega_2}(H(\theta))$ such that $\omega_1 \subseteq N$ 
and $N \prec (H(\theta),\in,\p)$. 
Suppose that $G$ is an $N$-generic filter on $\p$. 
Assume that $p \in G$, $\lambda \in N$, 
and $p$ forces that there is a countable subset $a$ of $\lambda$ 
such that for any countable set $c$ in $V$, $a \cap c$ is finite. 
Then there is a countable subset $b$ of $N \cap \lambda$ such that for any 
countable set $c$ in $N$, $b \cap c$ is finite.
\end{lemma}

\begin{proof}
Similar to the proof of Lemma 5.2.
\end{proof}

In previous work, 
Krueger \cite{jk8} used Woodin's theorem to prove that \textsf{PFA} implies 
the existence of 
stationarily many models which are internally club but not internally 
approachable. 
The proof involved specializing a certain tree of height and size $\omega_1$ 
which was built on a model $N$. 
Later, 
Viale-Weiss \cite{vialeweiss} expanded on Krueger's application of 
Woodin's theorem to produce $\omega_1$-guessing models.

The next result was originally proven in \cite{vialeweiss} for forcing posets 
which have the $\omega_1$-approximation property and models which are 
$\omega_1$-guessing. 
We have modified the argument to handle forcing posets which have the 
weak $\kappa$-approximation property and models which are 
indestructibly weakly $\kappa$-guessing.

\begin{proposition}
Let $\lambda \ge \omega_2$ be a regular cardinal and $A \subseteq \lambda$ a 
set of regular uncountable cardinals. 
Assume that $\p$ is a forcing poset which has the weak 
$\kappa$-approximation property for all $\kappa \in A$ and forces that  
$(2^{\lambda})^V$ has size $\omega_1$. 
Then there exist a set $w$ and a $\p$-name $\dot \q$ for an 
$\omega_1$-c.c.\ forcing poset satisfying: 
for any regular cardinal $\chi$ with $\p$, $\lambda$, $w$, and $\dot \q$ 
in $H(\chi)$, 
for any set $M \in P_{\omega_2}(H(\chi))$ such that 
$\omega_1 \subseteq M$ and 
$M \prec (H(\chi),\in,\p * \dot \q,\lambda,A,w)$, 
if there exists an $M$-generic filter on $\p * \dot \q$, 
then $M \cap H(\lambda)$ is indestructibly weakly $\kappa$-guessing for 
all $\kappa \in A \cap M$.
\end{proposition}

\begin{proof}
We begin by fixing a generic filter $G$ on $\p$, and analyze what 
happens in $V[G]$. 
We will define a tree $(T,<_T)$ in $V[G]$ of height and size $\omega_1$ 
which has at most $\omega_1$ many uncountable branches, together with 
subtrees $T^0$ and $T^1$, where $T^1$ has no uncountable branches. 
We then let $\q$ be the standard $\omega_1$-c.c.\ forcing poset 
for adding a specializing function to the tree $T^1$. 
The tree $T$ will be the disjoint sum of trees $(T_\kappa,<_\kappa)$, 
for $\kappa \in A$.

Working in $V[G]$, 
consider $\kappa \in A$. 
Since $\p$ has the $\kappa$-approximation property, it forces that 
$\cf(\kappa) > \omega$ by Lemma 3.8. 
As $\p$ collapses $\lambda$ to have size $\omega_1$, $\p$ forces that 
$\cf(\kappa) = \omega_1$. 
Fix a sequence $\langle \beta_i^\kappa : i < \omega_1 \rangle$ 
which is increasing and cofinal in $\kappa$. 
Let $T_\kappa$ denote the set of functions in $V$ whose domain 
is equal to $\beta_i^\kappa$ for some $i < \omega_1$, and whose 
range is a subset of $\lambda$. 
For $f$ and $g$ in $T_\kappa$, let $f <_\kappa g$ if $f$ is a proper subset of $g$. 
Clearly $(T_\kappa,<_\kappa)$ is a tree. 
Since $(2^\lambda)^V$ has size $\omega_1$ in $V[G]$, the 
tree $(T_\kappa,<_\kappa)$ has height and size $\omega_1$.

Let $B_\kappa$ be the set of all functions 
$f : \kappa \to \lambda$ such that for all 
$i < \omega_1$, $f \restriction \beta^\kappa_i \in T_\kappa$. 
Note that if $f \in B_\kappa$, then since $T_\kappa \subseteq V$, 
it follows that for all $\alpha < \kappa$, $f \restriction \alpha \in V$. 
Since $\p$ has the weak $\kappa$-approximation 
property, it follows that $B_\kappa \subseteq V$. 
So in fact $B_\kappa$ is the set of all functions in $V$ from 
$\kappa$ into $\lambda$. 
Since $(2^\lambda)^V$ has size $\omega_1$, 
$B_\kappa$ has size $\omega_1$. 
Note that if $X$ is an uncountable branch of $T_\kappa$, then 
$\bigcup X \in B_\kappa$. 
It follows that $T_\kappa$ has at most $\omega_1$ 
many uncountable branches.

Let $(T,<_T)$ be the disjoint sum of such trees. 
So elements of $T$ are pairs of the form $(\kappa,g)$, where 
$\kappa \in A$ and $g \in T_\kappa$, 
and $(\kappa_0,g) <_T (\kappa_1,h)$ if $\kappa_0 = \kappa_1$ and 
$g <_{\kappa_{0}} h$. 
Since $\lambda$ has size $\omega_1$ in $V[G]$, 
$T$ is a tree of height and size $\omega_1$. 
Since any uncountable branch of $T$ obviously yields an uncountable
branch in some tree $T_\kappa$, $T$ has at most 
$\omega_1$ many uncountable branches.

Let $B = \bigcup \{ B_\kappa : \kappa \in A \}$. 
Then $B$ has size at most $\omega_1$. 
So it is straightforward to define an injective function 
$g : B \to T$ such that for all $b \in B$, 
$g(b) = (\kappa,b \restriction \beta^\kappa_i)$ for some $i < \omega_1$, 
where $b \in B_\kappa$. 
For example, enumerate $B$ as $\langle d_i : i < \omega_1 \rangle$ 
and let $g(d_i) := (\kappa,d_i \restriction \beta^\kappa_i)$, where 
$d_i \in B_\kappa$.

Now define subtrees $T^0$ and $T^1$ of $T$ by 
$$
T^0 := \{ (\kappa,t) \in T : 
\exists b \in B_\kappa \ ( \ g(b) <_T (\kappa,t), \ t \subseteq b \ ) \}
$$
and 
$$
T^1 := T \setminus T^0.
$$
Then $T^1$ is a tree of height and size less than or equal to 
$\omega_1$.

We claim that $T^1$ has no uncountable branch. 
Suppose for a contradiction that $X$ is an uncountable branch of $T^1$. 
Then for some $\kappa \in A$, 
$$
b := \bigcup \{ f : (\kappa,f) \in X \}
$$
is in $B_\kappa$. 
Fix $i < \omega_1$ such that $g(b) = (\kappa,b \restriction \beta^\kappa_i)$. 
Since $X$ is uncountable, there is $j > i$ such that 
$(\kappa,b \restriction \beta^\kappa_j)$ is in $X$. 
Then 
$$
g(b) = (\kappa,b \restriction \beta^\kappa_i) 
<_T (\kappa,b \restriction \beta^\kappa_{j})
$$
and 
$b \restriction \beta^\kappa_{j} \subseteq b$. 
But then by the definition of $T^0$, 
$b \restriction \beta^\kappa_j \in T^0$. 
This contradicts that $X \subseteq T^1$.

In summary, $T^1$ is a tree with height and size less than or equal to $\omega_1$ 
which has no uncountable branches. 
Let $\q$ be the standard $\omega_1$-c.c.\ forcing poset which adds 
a specializing function to $T^1$.

\bigskip

Now let us go back and consider $\kappa \in A$ in $V[G]$. 
Define $g_\kappa : B_\kappa \to T_\kappa$ by letting $g_\kappa(b)$ be the unique $t$ 
such that $g(b) = (\kappa,t)$. 
Then $g_\kappa$ is injective, and 
$g_\kappa(b)$ is equal to $b \restriction \beta^\kappa_i$ for some $i < \omega_1$.

Define subtrees $T_\kappa^0$ and $T_\kappa^1$ of $T_\kappa$ by 
$$
T_\kappa^0 := \{ t : (\kappa,t) \in T^0 \}
$$
and 
$$
T_\kappa^1 := \{ t : (\kappa,t) \in T^1 \}.
$$
It is easy to check that 
$$
T_\kappa^0 = \{ t \in T_\kappa : \exists b \in B_\kappa \ 
( \ g_\kappa(b) <_\kappa t \subseteq b \ ) \},
$$
and 
$$
T_\kappa^1 = T_\kappa \setminus T_\kappa^0.
$$
Note that if $\dot h$ is a $\q$-name for the specializing function on $T^1$ which 
was added by $\q$, then $\q$ forces that 
the function $\dot h_\kappa : T_\kappa^1 \to \omega$ defined 
by $\dot h_\kappa(t) = \dot h(\kappa,t)$ is a specializing function on $T_\kappa^1$.

\bigskip

This completes our analysis of the forcing extension $V[G]$. 
Since $G$ was arbitrary, we can fix $\p$-names 
$\dot T$, $\dot{<_T}$, $\dot B$, $\dot g$, $\dot T^0$, 
$\dot T^1$, and $\dot \q$ such that $\p$ forces that these names 
satisfy the definitions which we made of these objects in $V^\p$ above. 
Moreover, we can fix a function which associates to each $\kappa \in A$ 
a set of $\p$-names 
$\langle \dot \beta^\kappa_i : i < \omega_1 \rangle$, 
$\dot T_\kappa$, $\dot{<_\kappa}$, $\dot B_\kappa$, $\dot g_\kappa$, $\dot T_\kappa^0$, 
and $\dot T_\kappa^1$ such that $\p$ forces that these 
objects satisfy the definitions which we made of them in $V^\p$ above. 
Finally, let $\dot h$ be a $\p * \dot \q$-name for the specializing function 
on $\dot T^1$ which is added by $\dot \q$, and fix 
a function which associates to each $\kappa \in A$ 
a $\p * \dot \q$-name for 
the function $\dot h_\kappa$ which specializes $\dot T_\kappa^1$. 
Now let $w$ be the set consisting of these finitely many names and functions.

\bigskip

Let $\chi$ be a regular cardinal such that 
$\p$, $\lambda$, $w$, and $\dot \q$ are in $H(\chi)$. 
Let $M \in P_{\omega_2}(H(\chi))$ be a set such that 
$\omega_1 \subseteq M$, $M \prec (H(\chi),\in,\p * \dot \q,\lambda,A,w)$, 
and there exists an $M$-generic filter $J$ on $\p * \dot \q$. 
We will prove that $M \cap H(\lambda)$ 
is indestructibly weakly $\kappa$-guessing for all $\kappa \in A \cap M$. 
Let $G := \{ p : \exists \dot q \ (p,\dot q) \in J \}$. 
It is easy to check that $G$ is an $M$-generic filter on $\p$.

Consider $\kappa \in A \cap M$. 
Let $\theta := \sup(M \cap \kappa)$. 
Suppose that $c : \theta \to \lambda$ is a function in an $\omega_1$-preserving 
generic extension $W$ 
such that for cofinally many 
$\alpha < \theta$, $c \restriction \alpha \in M \cap H(\lambda)$. 
We will prove that for some function $c^* : \kappa \to \lambda$ 
in $M \cap H(\lambda)$, $c^* \restriction \theta = c$. 
Note that since $\lambda$ is regular, any function mapping from 
$\kappa$ into $\lambda$ in $M$ is in $H(\lambda)$. 
So it suffices to find such a function $c^*$ in $M$.

For each $i < \omega_1$, let $\beta_i^\kappa$ be the unique ordinal such that 
for some $q \in G$, $q$ forces that $\dot \beta_i^\kappa = \check \beta_i^\kappa$. 
Since $G$ is a filter which is $M$-generic, it is straightforward to 
check that the sequence $\langle \beta_i^\kappa : i < \omega_1 \rangle$ 
is increasing and cofinal in $M \cap \kappa$.
For simplicity in notation, let us write $\beta_i := \beta_i^\kappa$ for all 
$i < \omega_1$.

Given any $\p$-name $\dot a$ in $M$, we can interpret $\dot a$ by $G$ 
as the 
set of $x \in M$ such that for some $p \in G \cap M$, 
$p \Vdash \check x \in \dot a$. 
This gives us objects 
$T_\kappa$, $<_\kappa$, $B_\kappa$, $g_\kappa$, $T_\kappa^0$, 
and $T_\kappa^1$ which 
interpret the $\p$-names $\dot T_\kappa$, $\dot{<_\kappa}$, $\dot B_\kappa$, 
$\dot g_\kappa$, $\dot T^0_\kappa$, and $\dot T^1_\kappa$. 
Similarly, interpret the $\p * \dot \q$-name $\dot h_\kappa$ as $h_\kappa$. 
Using the $M$-genericity of $G$ and $J$ and arguments similar to those of 
Lemma 5.2, the following facts can be easily checked:
\begin{enumerate}
\item $(T_\kappa,<_\kappa)$ is the tree of all functions in $M$ whose domain 
is equal to $\beta_i$ for some $i < \omega_1$, and mapping into $\lambda$, 
ordered by 
$f <_\kappa g$ if $f$ is a proper subset of $g$;
\item $B_\kappa$ is the set of all functions in $M$ of the form 
$b : \kappa \to \lambda$;
\item $g_\kappa : B_\kappa \to T_\kappa$ is injective and $g_\kappa(b) \subseteq b$ 
for all $b \in B_\kappa$;
\item $T_\kappa^0$ is the set of all $t \in T_\kappa$ such that for some 
$b \in B_\kappa$, $g_\kappa(b) <_\kappa t \subseteq b$;
\item $T_\kappa^1 = T_\kappa \setminus T_\kappa^0$;
\item $h_\kappa : T_\kappa^1 \to \omega$ is a function such that whenever 
$f <_\kappa g$ are in $T_\kappa^1$, then $h_\kappa(f) \ne h_\kappa(g)$.
\end{enumerate}

Recall that $c : \theta \to \lambda$ is a function in $W$ 
such that for cofinally many 
$\alpha < \theta$, $c \restriction \alpha \in M$. 
In particular, for all $i < \omega_1$, 
$c \restriction \beta_i \in T_\kappa$. 
Hence, the set $X := \{ c \restriction \beta_i : i < \omega_1 \}$ 
is an uncountable branch of $T_\kappa$ in $W$. 
Since the function $h_\kappa$ is injective on chains, 
it is injective on $X \cap T_\kappa^1$. 
As $h_\kappa$ maps into $\omega$, there must exist 
$\gamma < \omega_1$ such that for all 
$\gamma \le i < \omega_1$, $c \restriction \beta_i$ is in $T_\kappa^0$.

Now $T_\kappa^0$ is the set of all $t \in T_\kappa$ such that for some 
$b \in B_\kappa$, $g_\kappa(b) <_\kappa t \subseteq b$. 
So for all $\gamma \le i < \omega_1$, we can fix 
$b_i \in B_\kappa$ such that 
$$
g_\kappa(b_i) <_\kappa c \restriction \beta_i \subseteq b_i.
$$
In particular, $\dom(g_\kappa(b_i)) < \beta_i$, so fix 
$\zeta_i < i$ such that 
$\dom(g_\kappa(b_i)) = \beta_{\zeta_i}$. 
By Fodor's lemma applied in $W$, fix a stationary set 
$S \subseteq \omega_1 \setminus \gamma$ in $W$ 
and $\zeta < \omega_1$ such that 
for all $i \in S$, $\zeta_i = \zeta$, and hence 
$\dom(g_\kappa(b_i)) = \beta_\zeta$.

It immediately follows that for all $i \in S$, 
$$
g_\kappa(b_i) = b_i \restriction \beta_{\zeta_i} = 
b_i \restriction \beta_\zeta = 
(c \restriction \beta_i) \restriction \beta_\zeta = c \restriction \beta_\zeta.
$$
Thus, for all $i < j$ in $S$, 
$$
g_\kappa(b_i) = c \restriction \beta_\zeta = g_\kappa(b_j).
$$
But $g_\kappa$ is injective. 
Hence, for all $i < j$ in $S$, $b_i = b_j$. 
Let $b := b_i$ for some (any) $i \in S$. 
Then for all $i \in S$, $c \restriction \beta_i \subseteq b$. 
Since $S$ is cofinal in $\omega_1$, it follows that 
$c \subseteq b$. 
So $b \restriction \theta = c$. 
Since $b \in M$, we are done.
\end{proof}

\section{Namba forcing}

In order to apply the results from the previous section, we will need 
to find a forcing poset which satisfies instances of the weak approximation 
property and yet 
adds a countable set of ordinals which is 
not covered by any countable set in the ground model. 
In this section we will define a Namba forcing which satisfies 
these requirements.

Let us introduce some notation which we will use in this section. 
By a \emph{tree of finite sequences} we mean 
any set of finite sequences which is closed under initial segments. 
For finite sequences $\eta$ and $\nu$, 
we write $\eta \unlhd \nu$ to express that $\eta$ is an initial segment 
of $\nu$, and $\eta \lhd \nu$ to express that $\eta$ is a proper initial 
segment of $\nu$. 
We say that $\eta$ and $\nu$ are \emph{comparable} if either 
$\eta \unlhd \nu$ or $\nu \unlhd \eta$; otherwise, they are incomparable.

Let $T$ be a tree of finite sequences. 
The elements of $T$ are called \emph{nodes of $T$}. 
For $\eta \in T$, 
we write $\suc_T(\eta)$ for the set $\{ x : \eta \con x \in T \}$. 
A node $\eta$ of $T$ is a \emph{splitting node} if $|\suc_T(\eta)| > 1$. 
We let $T_\eta$ denote the set of nodes in $T$ which are comparable 
with $\eta$. 
Note that $T_\eta$ is a tree of finite sequences and is a subset of $T$.

\bigskip

For the remainder of the section, 
fix a sequence $\langle \kappa_n : n < \omega \rangle$ of 
regular cardinals greater than or equal to $\omega_2$. 
Note that we are not assuming anything about how the $\kappa_n$'s 
are ordered. 
For each $n < \omega$, we fix a $\kappa_n$-complete uniform ideal 
$I_n$ on $\kappa_n$.

\begin{definition}
Let $\p$ be the forcing poset whose conditions are trees 
of finite sequences $S$ satisfying:
\begin{enumerate}
\item there is $\eta^* \in S$ such that for all $\nu \in S$, $\nu$ and $\eta^*$ 
are comparable;
\item for all $\nu \in S$, if $\eta^* \unlhd \nu$ then 
$\nu$ is a splitting node and 
$\suc_{S}(\nu) \in I_{\lh(\nu)}^+$.\footnote{Recall that 
for any ideal $I$ on a set 
$X$, $I^+$ denotes the collection of subsets of $X$ which are not in $I$.} 
\end{enumerate}
Let $T \le S$ if $T \subseteq S$.
\end{definition}

The node $\eta^*$ described in Definition 6.1 is obviously unique. 
We call this node the \emph{stem of $S$}. 
Note that for all $\nu \in S$, if $\nu \lhd \eta^*$, 
then $\nu$ is not a splitting node.

We introduce reflexive and transitive 
relations $\le^*$ and $\le_n$, for each $n < \omega$, on $\p$. 
Define $T \le^* S$ if $T \le S$ and $S$ and $T$ have the same stem. 
For $n < \omega$, define $T \le_n S$ if $T \le S$ and $S$ and $T$ 
have the same nodes of length $n$. 
Note that $m \le n$ and $T \le_n S$ imply that $T \le_m S$.

A sequence $\langle T_n : n < \omega \rangle$ is a \emph{fusion sequence} 
if for all $n < \omega$, $T_{n+1} \le_n T_n$ and 
$T_{n+1} \le^* T_n$. 
If $\langle T_n : n < \omega \rangle$ is a fusion sequence, then it is easy to 
check that 
$$
\bigcap_{n < \omega} T_n = \bigcup_{n < \omega} \{ \nu \in T_n : 
\lh(\nu) = n \}.
$$

\begin{lemma}
Suppose that $\langle T_n : n < \omega \rangle$ is a fusion sequence. 
Let $T := \bigcap \{ T_n : n < \omega \}$. 
Then $T \in \p$, and for all $n < \omega$, 
$T \le^* T_n$ and $T \le_n T_n$.
\end{lemma}

\begin{proof}
Straightforward.
\end{proof}

The next lemma describes the process which we will use 
for constructing fusion sequences.

\begin{lemma}
Suppose that $\langle T_n : n < \omega \rangle$ is a sequence of 
conditions, where $T_0$ has a stem of length $m$, satisfying:
\begin{enumerate}
\item $T_0 = T_k$ for all $k \le m$;
\item for all $n \ge m$, for all $\nu \in T_n$ of length $n$, there is a 
set $\suc(\nu) \subseteq \suc_{T_n}(\nu)$ in $I_n^+$ such that 
$$
T_{n+1} = \bigcup \{ U(\nu,\xi) : \nu \in T_n, \ \lh(\nu) = n, \ 
\xi \in \suc(\nu) \},
$$
where 
$U(\nu,\xi) \le^* (T_n)_{\nu \con \xi}$ for all $\nu$ and $\xi$.
\end{enumerate}
Then $\langle T_n : n < \omega \rangle$ is a fusion sequence. 
Moreover, letting $T := \bigcap \{ T_n : n < \omega \}$, 
we have that for all $n < \omega$, 
the set 
$$
\{ U(\nu,\xi) : \nu \in T_n, \ \lh(\nu) = n, \ 
\xi \in \suc(\nu) \}
$$
is an antichain which is predense below $T$.
\end{lemma}

\begin{proof}
The proof is straightforward. 
For the second part, note that for a given $\nu$ in $T_n$ of length $n$ 
and distinct $\xi$ and $\gamma$ in $\suc(\nu)$, 
$U(\nu,\xi) \cap U(\nu,\gamma)$ is the set of initial segments of $\nu$. 
Thus, $U(\nu,\xi)$ and $U(\nu,\gamma)$ are incompatible. 
Suppose that $V \le T$. 
Then choosing any $\nu \in V$ with $\lh(\nu) = n$ and $\xi \in \suc_V(\nu)$, 
it is easy to see that $V_{\nu \con \xi}$ is below both $V$ and $U(\nu,\xi)$.
\end{proof}

The next four results follow from standard Namba forcing type arguments; 
also see \cite{namba} and \cite[Chapter XI]{shelah}.

\begin{lemma}
Let $D$ be a dense open subset of $\p$. 
Then for each $S \in \p$, there is $W \le^* S$ and $n < \omega$ which is 
greater than or equal to the length of the stem of $S$ 
such that for any $\nu \in W$ with $\lh(\nu) = n$, $W_\nu \in D$. 
\end{lemma}

\begin{proof}
For any condition $S$, let us say that $S$ is \emph{correct for $D$} if 
there exists $W \le^* S$ and $n < \omega$ which is 
greater than or equal to the length of the stem 
of $S$ such that for all $\nu \in W$ with 
$\lh(\nu) = n$, $W_\nu \in D$. 
Our goal is to show that every condition is correct for $D$. 
Note that if $W \le^* T$ and $W \in D$, then $T$ is correct 
for $D$, as witnessed by the number $n$ which is the length of 
the stem of $T$.

\bigskip

\noindent \emph{Claim}: If $T$ is not correct for $D$ 
and $\eta$ is the stem of $T$, 
then the set of $\gamma \in \suc_T(\eta)$ such that 
$T_{\eta \con \gamma}$ is correct for $D$ 
is in $I_{\lh(\eta)}$.

\bigskip

Let $A$ be the set of $\gamma \in \suc_T(\eta)$ 
such that $T_{\eta \con \gamma}$ is correct for $D$, 
and suppose for a contradiction that $A \notin I_{\lh(\eta)}$. 
Then $A \in I_{\lh(\eta)}^+$. 

For each $\gamma \in A$, 
fix $U(\gamma) \le^* T_{\eta \con \gamma}$ 
and $n_\gamma < \omega$ greater than or equal to $\lh(\eta)+1$ 
such that for any $\nu$ in $U(\gamma)$ with $\lh(\nu) = n_\gamma$, 
$U(\gamma)_\nu$ is in $D$. 
Since $I_{\lh(\eta)}$ is $\kappa_{\lh(\eta)}$-complete, 
we can fix $n < \omega$ 
such that the set $A_n := \{ \gamma \in A : n_\gamma = n \}$ is in 
$I_{\lh(\eta)}^+$. 
Now define $U := \bigcup \{ U(\gamma) : \gamma \in A_n \}$. 
Then $U$ is a condition and $U \le^* T$.

Note that $n$ is greater than the length of the stem $\eta$ of $U$. 
We claim that if $\nu \in U$ and $\lh(\nu) = n$, then 
$U_\nu \in D$. 
This implies that $T$ is correct for $D$, which is a contradiction. 
Suppose that $\nu \in U$ and $\lh(\nu) = n$. 
Fix $\gamma \in A_n$ such that $\nu \in U(\gamma)$. 
Then $U(\gamma)_\nu \in D$. 
But since 
$n$ is greater than the length of the stem of $U$, 
$U(\gamma)_\nu = U_\nu$, so $U_\nu \in D$. 
This completes the proof of the claim.

\bigskip

Now let $S$ be a condition, and we will prove that $S$ 
is correct for $D$. 
Suppose for a contradiction that it is not. 
We define a fusion sequence $\langle T_n : n < \omega \rangle$. 
Our inductive hypothesis is that if $\nu \in T_n$ and $\lh(\nu) = n$, 
then $(T_n)_\nu$ is not correct for $D$. 
Let $\eta$ be the stem of $S$, and let $n^* := \lh(\eta)$.

Define $T_m := S$ for all $m \le n^*$. 
Then for any $\nu \in T_m$ with $\lh(\nu) = m$, 
$\nu \unlhd \eta$, so $(T_m)_\nu = S$, which is not correct for $D$. 
Thus, the inductive hypothesis holds.

Let $m \ge n^*$, and assume that $T_m$ is defined and satisfies 
the inductive hypothesis. 
Consider any $\nu \in T_m$ with $\lh(\nu) = m$. 
Then by the inductive hypothesis, $(T_m)_\nu$ is not correct for $D$. 
Moreover, since the stem of $T_m$ is equal to $\eta$, which has length 
$n^*$, and $m \ge n^*$, it follows that $\nu$ is the stem 
of $(T_m)_\nu$. 
So by the claim, letting $B_\nu$ be the set of 
$\gamma \in \suc_{T_m}(\nu)$ such that 
$(T_m)_{\nu \con \gamma}$ is not correct for $D$, 
we have that $\suc_{T_m}(\nu) \setminus B_\nu$ is in $I_{m}$. 
Thus, $B_\nu \in I_{m}^+$.

Define 
$$
T_{m+1} := \bigcup \{ (T_m)_{\nu \con \gamma} : 
\nu \in T_m, \ \lh(\nu) = m, \ \gamma \in B_\nu \}.
$$
Then $T_{m+1} \le_m T_m$. 
Let us verify the inductive hypothesis for $T_{m+1}$. 
Suppose that $\xi \in T_{m+1}$ has length $m+1$. 
Then $\xi = \nu \con \gamma$ for some $\nu \in T_m$ with 
$\lh(\nu) = m$ and $\gamma \in B_\nu$. 
By definition, $(T_m)_{\nu \con \gamma}$ is not correct for $D$, and 
$(T_m)_{\nu \con \gamma} = (T_{m+1})_\xi$.

\bigskip

This completes the construction of the fusion sequence 
$\langle T_n : n < \omega \rangle$. 
Let $T := \bigcap \{ T_n : n < \omega \}$. 
Then $T \in \p$, and $T \le_n T_n$ for all $n < \omega$. 
Since $D$ is dense, we can fix $W \le T$ in $D$. 
Let $\nu$ be the stem of $W$, and fix $m$ with 
$\lh(\nu) = m$. 
Since the stem of $T$ is $\eta$, $m \ge \lh(\eta) = n^*$. 
Then $\nu \in T$, and hence $\nu \in T_m$. 
By the inductive hypothesis for $m$, we have that 
$(T_m)_\nu$ is not correct for $D$. 
But since $\nu$ is the stem of $W$, 
$W \le^* (T_m)_\nu$. 
Since $W \in D$, it follows that $(T_m)_\nu$ is correct for $D$, 
which is a contradiction.
\end{proof}

\begin{lemma}
Let $m < \omega$ and $\lambda$ be an ordinal, and suppose that 
for all $n \ge m$, $\lambda < \kappa_n$. 
Suppose that $S \in \p$ and the stem of $S$ has length at least $m$, 
and assume that $S \Vdash \dot \alpha < \lambda$. 
Then there is $T \le^* S$ such that $T$ decides the value of $\dot \alpha$.
\end{lemma}

\begin{proof}
Let $\eta$ be the stem of $S$ and let $n^* := \lh(\eta)$. 
Then $m \le n^*$. 
Let $D$ be the dense open set of conditions which are either 
incompatible with $S$, or below $S$ and decide the value of 
$\dot \alpha$. 
By Lemma 6.4, let $k$ be the smallest natural number 
greater than or equal to $n^*$ such that there exists 
$T \le^* S$ satisfying that for every node $\nu \in T$ with 
length $k$, $T_\nu \in D$. 
We claim that $k = n^*$. 
Then since the stem 
$\eta$ is a node of $T$ of length $n^*$, the claim implies that 
$T_\eta = T$ decides the value of $\dot \alpha$, and we are done.

Suppose for a contradiction that $k > n^*$. 
We will prove that there is $W \le^* S$ such that any node $\nu$ of 
$W$ of length $k-1$ satisfies that $W_\nu$ is in $D$. 
This will contradict the minimality of $k$ and finish the proof.

Consider any node $\nu$ of $T$ of length $k - 1$. 
Then since $k - 1 \ge n^*$, it follows that $\nu$ is a splitting node of $T$. 
By the choice of $T$ and $k$, we have that 
for any $\gamma \in \suc_T(\nu)$, 
$T_{\nu \con \gamma}$ is in $D$, 
and hence decides $\dot \alpha$ to be 
equal to some ordinal $\alpha(\nu,\gamma) < \lambda$. 
There are at most $\lambda$ many possibilities for $\alpha(\nu,\gamma)$. 
Since $k-1 \ge n^* \ge m$, $\lambda < \kappa_{k-1}$. 
As $I_{k-1}$ is $\kappa_{k-1}$-complete, we can find 
an ordinal $\alpha(\nu) < \lambda$ and a set 
$\suc(\nu) \subseteq \suc_T(\nu)$ which is in $I_{k-1}^+$ such that 
for all $\gamma \in \suc(\nu)$, $\alpha(\nu,\gamma) = \alpha(\nu)$.

Now define 
$$
W := \bigcup 
\{ T_{\nu \con \gamma} : \nu \in T, \ \lh(\nu) = k-1, \ 
\gamma \in \suc(\nu) \}.
$$
Then $W \le^* T \le^* S$, and so $W \le^* S$. 
To complete the proof, we show that for any node $\nu \in W$ with 
length $k-1$, $W_\nu \in D$. 
So let $\nu \in W$ have length $k-1$. 
Then for any 
$\gamma \in \suc_{W}(\nu)$, $W_{\nu \con \gamma} = 
T_{\nu \con \gamma}$ forces that $\dot \alpha$ is equal to 
$\alpha(\nu)$. 
It easily follows that 
$W_\nu$ forces that $\dot \alpha$ is equal to $\alpha(\nu)$, 
and hence $W_\nu \in D$.
\end{proof}

\begin{corollary}
For any $S \in \p$ and any statement $\varphi$ in the forcing language 
for $\p$, 
there is $T \le^* S$ which decides the truth value of $\varphi$.
\end{corollary}

\begin{proof}
Let $\dot \alpha$ be a $\p$-name which is forced to be equal to $1$ if 
$\varphi$ is true, and $0$ if $\varphi$ is false. 
Now apply Lemma 6.5 letting $m = 0$ and $\lambda = 2$. 
\end{proof}

\begin{proposition}
Let $\kappa := \lim \sup \{ \kappa_n : n < \omega \}$. 
Then for any regular cardinal $\lambda > \kappa$, $\p$ forces that 
$\cf(\lambda) > \omega$.
\end{proposition}

\begin{proof}
Let $S$ be a condition and suppose that $S$ forces that 
$\dot f : \omega \to \lambda$ is a function. 
We will find $T \le S$ which forces that $\dot f$ is bounded in $\lambda$. 
It suffices to find $T \le S$ and a set $A$ of size at most $\kappa$ 
such that $T$ forces that $\ran(\dot f) \subseteq A$.

Define a fusion sequence $\langle T_n : n < \omega \rangle$ as follows. 
Fix $m < \omega$ which is greater than or equal to 
the length of the stem of $S$ such that 
$\kappa = \sup \{ \kappa_n : m \le n \}$. 
Fix $\eta \in S$ with $\lh(\eta) = m$. 
For all $n \le m$, let $T_n := S_\eta$.

Let $n \ge m$, and assume that $T_k$ is defined for all $k \le n$. 
We will define $T_{n+1}$. 
Consider a node $\nu \in T_n$ of length $n$. 
Let $D_n$ be the dense open set of conditions which are either incompatible 
with $S$, or below $S$ and decide the value of $\dot f(n-m)$. 
Applying Lemma 6.4 to $(T_n)_\nu$ and $D_n$, we can 
fix $n(\nu) < \omega$ 
greater than or equal to $n$ 
and $U(\nu) \le^* (T_n)_\nu$ such that for 
any node $\sigma$ in $U(\nu)$ of 
length $n(\nu)$, $U(\nu)_\sigma$ is in $D_n$. 
Let $\xi(\nu,\sigma)$ denote the ordinal such that 
$U(\nu)_\sigma$ forces that $\dot f(n-m) = \xi(\nu,\sigma)$.

Define 
$$
T_{n+1} := \bigcup \{ U(\nu) : \nu \in T_n, \ \lh(\nu) = n \}.
$$
Then $T_{n+1} \le_n T_n$. 
Let 
$$
A_n := \{ \xi(\nu,\sigma) : \nu \in T_{n}, \ \lh(\nu) = n, \ 
\sigma \in U(\nu), \ \lh(\sigma) = n(\nu) \}.
$$
It is easy to check that $T_{n+1} \Vdash \dot f(n-m) \in \check A_n$.

We claim that $|A_n| \le \kappa$. 
If $\nu \in T_n$ and $\lh(\nu) = n$, 
then $\eta \unlhd \nu$. 
Since $\kappa_k \le \kappa$ for all $k$ with $\lh(\eta) = m \le k \le n$, 
there are at most $\kappa$ many $\nu \in T_n$ with $\lh(\nu) = n$. 
If $\sigma \in U(\nu)$ and $\lh(\sigma) = n(\nu)$, then since 
$\kappa_k \le \kappa$ for all $k$ with $n < k \le n(\nu)$, there are at most 
$\kappa$ many possibilities for $\sigma$. 
It follows that $|A_n| \le \kappa$.

This completes the construction. 
Let $T := \bigcap \{ T_n : n < \omega \}$ and 
$A := \bigcup \{ A_n : n < \omega \}$. 
Then $T \le S$ and $T$ forces that $\ran(\dot f) \subseteq A$. 
Since $|A_n| \le \kappa$ for all $n < \omega$, 
$|A| \le \kappa < \lambda$.
\end{proof}

\begin{proposition}
The forcing poset $\p$ forces that there exists a countable set $a$ 
such that for any set $c$ in the ground model with size less than 
$\lim \inf \{ \kappa_n : n < \omega \}$ in the ground model, 
$a \cap c$ is finite.

In particular, $\p$ forces that there is a countable set which 
is not covered by any countable set in the ground model.
\end{proposition}

\begin{proof}
Let $\dot F$ be a $\p$-name for a function 
such that $\p$ forces that for all 
$n < \omega$, $\dot F(n) = \alpha$ iff there is $S \in \dot G_\p$ 
with stem $\eta$ of length greater than $n$ such that 
$\eta(n) = \alpha$. 
It is straightforward to check that $\p$ forces that 
$\dot F$ is well-defined and is a 
total function on $\omega$. 
Let $\dot a$ be a $\p$-name for the range of $\dot F$.

Let $c$ be a set with size less than $\lim \inf \{ \kappa_n : n < \omega \}$ 
and $S \in \p$, and we will find $T \le S$ 
which forces that $\dot a \cap c$ is finite. 
Fix an integer $m$ which is greater than or equal to the length of the stem of $S$ 
such that for all $n \ge m$, $|c| < \kappa_n$. 
We define a fusion sequence $\langle T_n : n < \omega \rangle$. 
Fix $\eta \in S$ with $\lh(\eta) = m$, and let 
let $T_n := S_\eta$ for all $n \le m$.

Let $n \ge m$ be given, and assume that $T_n$ is defined. 
For each $\nu \in T_n$ of length $n$, 
define $\suc(\nu) := \suc_{T_n}(\nu) \setminus c$. 
Since $|c| < \kappa_n$ and $I_n$ is $\kappa_n$-complete, 
$\suc(\nu) \in I_n^+$. 
Note that for all $\gamma \in \suc(\nu)$, 
$(T_n)_{\nu \con \gamma}$ forces that 
$\dot F(n) = \gamma \notin c$. 
Now let 
$$
T_{n+1}:= \bigcup \{ (T_n)_{\nu \con \gamma} : 
\nu \in T_n, \ \lh(\nu) = n, \ \gamma \in \suc(\nu) \}.
$$

This completes the construction of $\langle T_n : n < \omega \rangle$. 
Let $T := \bigcap \{ T_n : n < \omega \}$. 
We claim that $T$ forces that $\dot a \cap c \subseteq \ran(\eta)$. 
First, since $\eta$ is the stem of $T$, 
$T$ forces that $\dot F \restriction m = \eta$.

Secondly, consider $n \ge m$. 
Then by Lemma 6.3, the set 
$$
\{ (T_n)_{\nu \con \gamma} : \nu \in T_n, \ \lh(\nu) = n, \ 
\gamma \in \suc(\nu) \}
$$
is predense below $T$. 
But any condition in this set forces that $\dot F(n) \notin c$. 
Therefore, $T$ forces that $\dot F(n) \notin c$. 
It follows that $T$ forces that anything which is in both 
$\dot a = \ran(\dot F)$ and in $c$ is in the range of 
$\dot F \restriction m = \eta$.

For the second statement, 
note that $\omega < \lim \inf \{ \kappa_n : n < \omega \}$. 
\end{proof}

\begin{corollary}
The forcing poset $\p$ does not have the 
$\omega_1$-approximation property.
\end{corollary}

\begin{proof}
Immediate from Proposition 6.8.
\end{proof}

It is not necessarily the case that the Namba forcing 
$\p$ satisfies the weak $\omega_1$-approximation property. 
For example, assume \textsf{CH} and $\kappa_n = \omega_2$ for 
all $n < \omega$. 
Then by \cite{namba} and \cite[Chapter XI]{shelah}, $\p$ does not 
add any reals. 
On the other hand, since $\p$ adds a countable cofinal subset 
of $\omega_2$, it collapses $\omega_2$, and therefore 
adds a new subset of $\omega_1$. 
But since $\p$ does not add reals, all of the proper initial segments 
of the characteristic function of the new subset of $\omega_1$ are 
in the ground model. 
Therefore, $\p$ does not have the weak $\omega_1$-approximation 
property.

In order to find a Namba forcing which does have the 
weak $\omega_1$-approximation property, we need to make 
an additional assumption about the ideals $I_n$, for $n < \omega$.

\begin{assumption}
For each $n < \omega$, there is a regular uncountable cardinal 
$\mu_n \le \kappa_n$ and a set $P_n \subseteq I_n^+$ 
satisfying:
\begin{enumerate}
\item for every $A \in I_n^+$, there is $B \in P_n$ such that 
$B \subseteq A$;
\item if $\langle A_i : i < \delta \rangle$ is a $\subseteq$-decreasing 
sequence of sets in $P_n$, where $\delta < \mu_n$, 
then there is $B \in P_n$ such that 
$B \subseteq \bigcap \{ A_i : i < \delta \}$.
\end{enumerate}
\end{assumption}

In the next section we will describe a model in which Assumption 6.10 holds.

\begin{definition}
Let $\p'$ denote the suborder of $\p$ consisting of conditions $S$ 
satisfying that for any splitting node $\nu \in S$, 
$\suc_S(\nu) \in P_{\lh(\nu)}$.
\end{definition}

\begin{lemma}
The set $\p'$ is dense in $\p$. 
In fact, for all $S \in \p$, there is $T \in \p'$ such that $T \le^* S$.
\end{lemma}

\begin{proof}
Let $m$ be the length of the stem of $S$. 
Define a fusion sequence $\langle T_n : n < \omega \rangle$ as follows. 
Let $T_n := S$ for all $n \le m$. 

Let $n \ge m$, and assume that $T_n$ is defined. 
Let $\nu$ be a node of $T_n$ with length $n$. 
Pick a set $\suc(\nu)$ in $P_n$ such that 
$\suc(\nu) \subseteq \suc_{T_n}(\nu)$. 
Define 
$$
T_{n+1} = \bigcup \{ (T_n)_{\nu \con \gamma} : 
\nu \in T_n, \ 
\lh(\nu) = n, \ \gamma \in \suc(\nu) \}.
$$

Now let $T := \bigcap \{ T_n : n < \omega \}$. 
Then $T \in \p$, and for all $n < \omega$, 
$T \le^* T_n$ and $T \le_n T_n$. 
In particular, the stem of $T$ is equal to the stem of $S$. 

To show that $T \in \p'$, let $\nu \in T$ be a splitting node, and let 
$n$ be the length of $\nu$. 
Then $m \le n$ and $\nu \in T_{n+1}$. 
Since $T \le_{n+1} T_{n+1}$, $T$ and $T_{n+1}$ have the same 
nodes of length $n+1$. 
In particular, $\suc_{T}(\nu) = \suc_{T_{n+1}}(\nu) = 
\suc(\nu) \in P_n$. 
So $T \in \p'$.
\end{proof}

\begin{lemma}
Let $m < \omega$, and suppose that $\delta$ is a limit ordinal such that 
$\delta < \mu_n$ for all $n \ge m$. 
Let $\langle T_i : i < \delta \rangle$ be a $\le^*$-descending 
sequence of conditions in $\p'$ such that the stem of 
$T_0$ has length at least $m$. 
Then there is $W \in \p'$ such that $W \le^* T_i$ for all $i < \delta$.
\end{lemma}

\begin{proof}
Let $T := \bigcap \{ T_i : i < \delta \}$. 
Then $T \subseteq T_i$ for all $i < \delta$. 
Since each condition $T_i$ has the same stem $\eta^*$, 
we also have that $\eta^* \in T$ and every node in $T$ is comparable with $\eta^*$.

Consider a node $\nu \in T$ such that $\eta^* \unlhd \nu$, and we will 
show that $\suc_{T}(\nu) \in I_{\lh(\nu)}^+$. 
As $\lh(\nu) \ge m$, $\delta < \mu_{\lh(\nu)}$. 
Since $\nu \in T$, for all $i < \delta$, $\nu \in T_i$. 
As $T_j \le^* T_i$ for all $i < j < \delta$, we have that the sequence  
$$
\langle \suc_{T_i}(\nu) : i < \delta \rangle
$$ 
is a $\subseteq$-decreasing sequence of sets in $P_{\lh(\nu)}$. 
Since $\delta < \mu_{\lh(\nu)}$, 
Assumption 6.10(2) implies that there is $A \in P_{\lh(\nu)}$ such that 
$A \subseteq \bigcap \{ \suc_{T_i}(\nu) : i < \delta \}$. 
In particular, this intersection is in $I_{\lh(\nu)}^+$. 
Therefore, 
$$
\suc_{T}(\nu) = \bigcap \{ \suc_{T_i}(\nu) : i < \delta \} 
\in I_{\lh(\nu)}^+.
$$

It follows that $T$ is a condition and $T \le^* T_i$ for all $i < \delta$. 
Now apply Lemma 6.12 to find $W \in \p'$ such that $W \le^* T$. 
Then $W \le^* T_i$ for all $i < \delta$.
\end{proof}

\begin{proposition}
Let $\kappa := \lim \inf \{ \mu_n : n < \omega \}$. 
Then $\p$ does not add any bounded subsets of $\kappa$.
\end{proposition}

\begin{proof}
Let $\lambda < \kappa$ be a limit ordinal, $S \in \p$, and 
suppose that $S \Vdash \dot a \subseteq \lambda$. 
We will find a condition below $S$ which decides $\dot a$.

Fix $m$ such that for all $n \ge m$, $\lambda < \mu_n$, and moreover, 
$m$ is greater than or equal to the length of the stem of $S$. 
We will define a $\le^*$-descending sequence 
$\langle T_i : i \le \lambda \rangle$ of conditions in $\p'$. 

Fix $\eta \in S$ such that $\lh(\eta) = m$, and fix $T_0 \in \p'$ such that 
$T_0 \le^* S_\eta$. 
Note that $\eta$ is the stem of $T_0$.

Suppose that $\beta < \lambda$ and 
$\langle T_i : i \le \beta \rangle$ is defined. 
Applying Corollary 6.6 and Lemma 6.12, 
fix a condition $T_{\beta+1}$ in $\p'$ such that 
$T_{\beta+1} \le^* T_\beta$ and $T_{\beta+1}$ decides 
whether or not $\beta$ is in $\dot a$.

Assume that $\delta \le \lambda$ is a limit ordinal and 
$\langle T_i : i < \delta \rangle$ is defined. 
Recall that for all $n \ge m$, $\delta \le \lambda < \mu_n$, and 
the length of the stem of $T_0$ is equal to $m$. 
By Lemma 6.13, we can fix $T_\delta \in \p'$ such that for all 
$i < \delta$, $T_\delta \le^* T_i$.

This completes the construction of the sequence. 
Let 
$$
b := \{ \beta < \lambda : T_{\beta+1} \Vdash \beta \in \dot a \}.
$$
Then $T_\lambda \le S$ and $T_\lambda$ forces that $\dot a$ is equal to $b$.
\end{proof}

We now turn to showing that the forcing poset $\p$ has the 
weak $\mu$-approximation property, for all regular uncountable 
cardinals $\mu < \lim \inf \{ \mu_n : n < \omega \}$. 
This will follow easily from the next proposition, which describes 
a stronger property of $\p$.

\begin{proposition}
Suppose that $\mu$ is a regular uncountable cardinal and 
$\mu < \lim \inf \{ \mu_n : n < \omega \}$. 
Then $\p$ forces that whenever $\langle a_i : i < \mu \rangle$ 
is a sequence of sets such that for all $i < \mu$, $a_i \in V$, 
then there is an unbounded set 
$X \subseteq \mu$ such that the sequence 
$\langle a_i : i \in X \rangle$ is in $V$.
\end{proposition}

\begin{proof}
Fix $m < \omega$ such that for all $n \ge m$, $\mu < \mu_n$. 
Suppose that $S$ forces that $\langle \dot a_i : i < \mu \rangle$ 
is a sequence such that for all $i < \mu$, $\dot a_i \in V$. 

We will define a $\le^*$-descending sequence 
$\langle T_i : i \le \mu \rangle$ of conditions in $\p'$. 
Fix $\eta \in S$ which extends the stem of $S$ and has length 
at least $m$. 
Fix $T_0 \in \p'$ such that $T_0 \le^* S_\eta$. 
Then the length of the stem of $T_0$ is at least $m$.

Let $\beta < \mu$, and suppose that 
$\langle T_i : i \le \beta \rangle$ is defined. 
Let $D_\beta$ be the dense open set of conditions which are either 
incompatible with $T_\beta$, or below $T_\beta$ and 
decide the value of $\dot a_\beta$. 
By Lemmas 6.4 and 6.12, we can fix $T_{\beta+1} \le^* T_\beta$ 
in $\p'$ and 
$n_\beta < \omega$ which is greater than or equal 
to the length of the stem of $T_\beta$ such that for any 
$\nu \in T_{\beta+1}$ with $\lh(\nu) = n_\beta$, 
$(T_{\beta+1})_\nu \in D_\beta$.

Assume that $\delta \le \mu$ is a limit ordinal and the sequence 
$\langle T_i : i < \delta \rangle$ is defined. 
Then for all $n \ge m$, $\delta \le \mu < \mu_n$. 
Also, the length of the stem of $T_0$ is at least $m$. 
By Lemma 6.13, fix $T_\delta \in \p'$ such that 
$T_\delta \le^* T_i$ for all $i < \delta$.

This completes the construction of the sequence 
$\langle T_i : i \le \mu \rangle$. 
Let $T := T_\mu$. 
Then $T \le S$.

For each $\beta < \mu$, $n_\beta < \omega$. 
Since $\mu$ is regular and uncountable, we can find $n < \omega$ 
such that the set $X := \{ \beta < \mu : n_\beta = n \}$ 
is unbounded in $\mu$.

Fix $\xi \in T$ such that $\lh(\xi) = n$, and let $W := T_\xi$. 
Then $W \le S$. 
We claim that $W$ forces that the sequence 
$\langle \dot a_i : i \in X \rangle$ is in $V$.

Consider $\beta \in X$. 
Since $\xi \in T$, $\xi \in T_{\beta+1}$ and $\lh(\xi) = n = n_\beta$. 
By the choice of $T_{\beta+1}$ and $n_\beta$, 
$(T_{\beta+1})_\xi \in D_\beta$. 
By the definition of $D_\beta$ and since 
$(T_{\beta+1})_\xi \le T_\beta$, there is a set $b_\beta$ 
such that $(T_{\beta+1})_\xi \Vdash \dot a_\beta = \check{b}_\beta$. 
But $T \le T_{\beta+1}$ implies that 
$W = T_\xi \le (T_{\beta+1})_\xi$. 
Hence, $W \Vdash \dot a_\beta = \check{b}_\beta$. 

It follows that $W$ forces that the sequence 
$\langle \dot a_\beta : \beta \in X \rangle$ 
is equal to the sequence $\langle b_\beta : \beta \in X \rangle$. 
Since the latter sequence is in $V$, we are done.
\end{proof}

\begin{corollary}
Suppose that $\mu$ is a regular uncountable cardinal such that  
$\mu < \lim \inf \{ \mu_n : n < \omega \}$. 
Then $\p$ has the weak $\mu$-approximation property.
\end{corollary}

\begin{proof}
Suppose that $S$ forces that $\dot f : \mu \to On$ is a function 
such that for all $\alpha < \mu$, $\dot f \restriction \alpha \in V$. 
Consider the sequence 
$\langle \dot f \restriction \alpha : \alpha < \mu \rangle$. 
Then $S$ forces that every member of this sequence is in $V$. 
By Proposition 6.15, there exist $T \le S$, an unbounded set 
$X \subseteq \mu$, and a sequence 
$\langle g_\alpha : \alpha \in X \rangle$ such that 
$T$ forces that $\dot f \restriction \alpha = g_\alpha$ for all $\alpha \in X$. 
In particular, for each $\alpha \in X$, $g_\alpha$ a function from 
$\alpha$ to $On$, and for all $\alpha < \beta$ in $X$, 
$g_\alpha = g_\beta \restriction \alpha$. 
It follows that 
$g := \bigcup \{ g_\alpha : \alpha \in X \}$ is a total function on $\mu$ 
and $T \Vdash \dot f = \check g$.
\end{proof}

\section{The main theorem}

We are now ready to complete the main result of the paper, which is to 
construct a model in which there are stationarily many 
$N \in P_{\omega_2}(H(\aleph_{\omega+1}))$ 
such that $N$ is indestructibly weakly guessing, has uniform cofinality $\omega_1$, 
and is not internally unbounded.

We will use the following well-known facts.

\begin{thm}[Larson]
Martin's maximum is preserved after forcing with any 
$\omega_2$-directed closed forcing poset.
\end{thm}

\begin{proof}
See \cite[Theorem 4.3]{larson}.
\end{proof}

\begin{thm}[Laver]
Let $\mu < \kappa$, where $\mu$ is regular uncountable and $\kappa$
is a measurable cardinal. 
Then $\coll(\mu,<\! \kappa)$ forces that there exists a $\kappa$-complete 
uniform ideal $I$ on $\kappa = \mu^+$ and a set $P \subseteq I^+$ satisfying:
\begin{enumerate}
\item for all $A \in I^+$, there is $B \in P$ such that $B \subseteq A$;
\item whenever 
$\langle B_i : i < \delta \rangle$ is a $\subseteq$-decreasing sequence 
of sets in $P$, where $\delta < \mu$, then 
$\bigcap \{ B_i : i < \delta \} \in I^+$.
\end{enumerate}
\end{thm}

\begin{proof}
By \cite[Theorem 7.6]{foreman}, there is a $\kappa$-complete 
uniform ideal $I$ on $\kappa = \mu^+$ 
such that the forcing poset $P(\kappa) / I$ has a dense, $\mu$-closed subset. 
A straightforward argument using the $\kappa$-completeness of $I$ shows 
that this implies that the forcing poset $(I^+,\subseteq)$ has a dense, $\mu$-closed subset.
\end{proof}

\begin{thm}[Hamkins]
If $\dot \q$ is an $\add(\omega)$-name for an $\omega_1$-closed 
forcing poset, then $\add(\omega) * \dot \q$ has the 
$\omega_1$-approximation property.
\end{thm}

\begin{proof}
See \cite[Lemma 13]{hamkins}.
\end{proof}

We start with a ground model $V$ in which there is a supercompact cardinal 
$\kappa$ and an increasing sequence 
$\langle \kappa_n : n < \omega \rangle$ of measurable 
cardinals which are above $\kappa$.
Let $\kappa_{-1} := \omega_1$.

Let $\mathbb M$ be the standard forcing poset 
which collapses $\kappa$ to become $\omega_2$ and 
forces Martin's maximum (see \cite[Section 1]{fms}). 
Let $G$ be a generic filter on $\mathbb{M}$. 
Then in $V[G]$, $\kappa = \omega_2$ and Martin's maximum holds. 
Since $\mathbb{M}$ has size $\kappa$ in $V$, in $V[G]$ we still have that 
$\kappa_n$ is measurable for all $n < \omega$. 

In $V[G]$, let $\langle \p_n, \dot \q_m : n \le \omega, \ m < \omega \rangle$ 
be the full support forcing iteration such that for all $n < \omega$, 
$$
\Vdash_{\p_n} \dot \q_n = \coll(\kappa_{n-1}^+,<\! \kappa_n).
$$
By standard arguments, $\p_\omega$ is $\omega_2$-directed closed, and 
for each $n < \omega$, $\p_\omega$ forces that 
$\kappa_n = \omega_{2n+3}$.

Let $H$ be a $V[G]$-generic filter on $\p_\omega$. 
Since $\p_\omega$ is $\omega_2$-directed closed, Theorem 7.1 implies 
that Martin's maximum holds in $V[G][H]$.

Consider $n < \omega$. 
Then in $V[G]$, 
$\p_\omega$ is forcing equivalent to a three-step forcing iteration of the form 
$$
\p_n * \coll(\kappa_{n-1}^+,<\! \kappa_n) * \p^n,
$$
where $|\p_n| < \kappa_n$, and $\p^n$ is forced to be $\kappa_n^+$-closed. 
Let $H_n * H(n) * H^n$ be a $V[G]$-generic filter for the above forcing poset 
such that 
$$
V[G][H] = V[G][H_n][H(n)][H^n].
$$

Since $\p_n$ has size less than $\kappa_n$, in $V[G][H_n]$, $\kappa_n$ is 
a measurable cardinal. 
And $H(n)$ is a $V[G][H_n]$-generic filter on 
$\coll(\kappa_{n-1}^+,<\! \kappa_n) = \coll(\omega_{2n+2},<\! \kappa_n)$. 
By Theorem 7.2, in $V[G][H_n][H(n)]$ 
there exists a $\kappa_n$-complete uniform ideal 
$I_n$ on $\kappa_n = \omega_{2n+3}$ 
and a set $P_n \subseteq I_n^+$ satisfying:
\begin{enumerate}
\item for all $A \in I_n^+$, there is $B \in P_n$ such that $B \subseteq A$;
\item whenever 
$\langle B_i : i < \delta \rangle$ is a $\subseteq$-decreasing sequence 
of sets in $P_n$, where $\delta < \omega_{2n+2}$, then 
$\bigcap \{ B_i : i < \delta \} \in I_n^+$.
\end{enumerate}
Finally, since $\p^n$ is $\kappa_n^+$-closed in $V[G][H_n][H(n)]$, 
it does not add any new subsets of $\kappa_n$, and therefore $I_n$ and $P_n$ 
satisfy exactly the same properties in the final model 
$V[G][H] = V[G][H_n][H(n)][H^n]$.

Let $W := V[G][H]$. 
Then $W$ satisfies the following statements:
\begin{enumerate}
\item Martin's maximum holds;
\item for all $n < \omega$, $\kappa_n = \omega_{2n+3}$;
\item for all $n < \omega$, $I_n$ is a $\kappa_n$-complete uniform ideal 
on $\kappa_n$;
\item for all $n < \omega$, $P_n \subseteq I_n^+$ satisfies 
\begin{enumerate}
\item for all $A \in I_n^+$, there is $B \in P_n$ such that $B \subseteq A$;
\item whenever 
$\langle B_i : i < \delta \rangle$ is a $\subseteq$-decreasing sequence 
of sets in $P_n$, where $\delta < \omega_{2n+2}$, then 
$\bigcap \{ B_i : i < \delta \} \in I_n^+$;
\end{enumerate}
\item $\lim \inf \{ \kappa_n : n < \omega \} = \lim \sup \{ \kappa_n : n < \omega \} 
= \aleph_\omega$.
\end{enumerate}

\bigskip

Working in the model $W$, we let $\p$ be the Namba forcing defined 
in Definition 6.1 using the sequence of cardinals 
$\langle \kappa_n : n < \omega \rangle = 
\langle \omega_{2n+3} : n < \omega \rangle$ 
and the sequence $\langle I_n : n < \omega \rangle$ of ideals just 
described. 
By Proposition 6.7, for any regular cardinal $\lambda > \aleph_\omega$, 
$\p$ forces that $\cf(\lambda) > \omega$.

Observe that Assumption 6.10 is satisfied, 
where we let $\mu_n := \omega_{2n+2}$ for all $n < \omega$. 
And $\lim \inf \{ \mu_n : n < \omega \} = \aleph_\omega$. 
It follows by Proposition 6.14 that $\p$ does not add any bounded subsets 
of $\aleph_\omega$. 
In particular, $\p$ preserves stationary subsets of $\omega_1$. 
By Corollary 6.16, $\p$ has the weak $\omega_n$-approximation property, 
for all $1 \le n < \omega$.

\bigskip

In $W^\p$ consider 
$$
\mathbb C := \add(\omega) * 
\coll(\omega_1,(2^{\aleph_{\omega+1}})^W).
$$
By Theorem 7.3 and Lemma 3.4, 
$\mathbb C$ has the weak $\omega_n$-approximation 
property in $W^\p$, for all $1 \le n < \omega$. 
It easily follows that $\p * \dot{\mathbb{C}}$ has the 
weak $\omega_n$-approximation property in $W$, for all $1 \le n < \omega$.
So $\p * \dot{\mathbb C}$ satisfies all of the assumptions 
of Proposition 5.4, where $\lambda = \aleph_{\omega+1}$ 
and $A = \{ \omega_n : 1 \le n < \omega \}$. 
Fix a set $w$ and a $\p * \dot{\mathbb C}$-name 
$\dot \q$ 
which satisfy the conclusion of Proposition 5.4.

Since $\mathbb C$ is proper in $W^\p$, $\p * \dot{\mathbb{C}}$ 
preserves stationary subsets of $\omega_1$. 
As $\dot \q$ is forced to be $\omega_1$-c.c., 
$\p * \dot{\mathbb{C}} * \dot \q$ 
preserves stationary subsets of $\omega_1$.

\bigskip

We are ready to complete the proof. 
Working in $W$, 
fix a regular cardinal $\chi > \aleph_{\omega+1}$ such that 
$w$ and $\p * \dot{\mathbb C} * \dot \q$ are in $H(\chi)$. 
Let $F : H(\aleph_{\omega+1})^{<\omega} \to H(\aleph_{\omega+1})$ 
be a function.

Using the fact that Martin's maximum holds in $W$, 
apply Theorem 5.1 to find a set 
$N \in P_{\omega_2}(H(\chi))$ such that 
$\omega_1 \subseteq N$, 
$N$ is closed under $F$, 
$N \prec (H(\chi),\in,\p*\dot{\mathbb{C}}*\dot \q,\aleph_{\omega+1},w)$, 
and there exists an $N$-generic filter $J$ on 
$\p*\dot{\mathbb{C}}*\dot \q$. 
Let $M := N \cap H(\aleph_{\omega+1})$.
Note that $M$ is closed under $F$.

By Proposition 5.4, $M$ is indestructibly weakly $\omega_n$-guessing, 
for all $1 \le n < \omega$. 
Hence, $M$ is indestructibly weakly guessing. 
Also note that since $\p*\dot{\mathbb{C}}*\dot \q$ forces that 
$\cf(\aleph_{\omega+1}) > \omega$, $\sup(M \cap \aleph_{\omega+1})$ 
has cofinality $\omega_1$ by an argument similar to the proof of Lemma 5.2. 
By Lemma 3.2, it follows that $M$ has uniform cofinality $\omega_1$. 
By Lemma 5.2 and Proposition 6.8, there is a countable subset of 
$M \cap \aleph_\omega$ which is not covered by any countable set in $M$. 
Since $\aleph_{\omega} \in M$, this set is bounded below 
$\sup(M \cap On)$.\footnote{Observe that 
by Lemma 5.3 and Proposition 6.8, $M$ is not $\omega_1$-guessing.}

\bigskip

There are many possible variations of the above construction. 
For example, let $\kappa$ be supercompact and $\lambda > \kappa$ 
be measurable. 
Let $G$ be a generic filter on the forcing poset 
$\mathbb{M}$ described above, 
and let $H$ be a $V[G]$-generic filter on $\coll(\omega_2,<\! \lambda)$. 
Then in $V[G][H]$, Martin's maximum holds and there is an 
$\omega_3$-complete uniform ideal 
$I$ on $\lambda = \omega_3$ and a set $P \subseteq I^+$ which 
satisfy Assumption 6.10 for $\mu = \omega_2$.

Let $\p$ be the Namba forcing from Definition 6.1, where we let 
$\kappa_n = \omega_3$ and $I_n = I$ for all $n < \omega$. 
Then $\p$ does not add any subsets of $\omega_1$, preserves $\omega_2$, 
changes the cofinality of $\omega_3$ to $\omega$, and preserves the 
uncountable cofinality of any regular cardinal greater than $\omega_3$.
Moreover, $\p$ has the weak $\omega_1$-approximation property.

Arguing as above, in $V[G][H]$ we have that for any large enough 
regular cardinal $\theta$, there are stationarily 
many $N \in P_{\omega_2}(H(\theta))$ such that $\omega_1 \subseteq N$, 
$N$ is indestructibly weakly $\omega_1$-guessing, 
$\cf(\sup(N \cap \omega_3)) = \omega$, 
and for all regular uncountable cardinals 
$\lambda \in N \cup \{ \theta \}$ different from 
$\omega_3$, $\cf(\sup(N \cap \lambda)) = \omega_1$.

As another example, assume that $\kappa < \lambda$, where $\kappa$ 
is supercompact and $\lambda$ is measurable. 
Let $G$ be a generic filter on $\mathbb{M}$. 
In $V[G]$, $\lambda$ is measurable, so we can let $I$ the dual ideal of a normal 
ultrafilter on $\lambda$. 
Let $\kappa_n := \lambda$ and $I_n := I$ for all $n < \omega$. 
Then in $V[G]$, 
we have that for any large enough 
regular cardinal $\theta$, there are stationarily 
many $N \in P_{\omega_2}(H(\theta))$ such that $\omega_1 \subseteq N$, 
$N$ is indestructibly weakly $\mu$-guessing for all regular uncountable cardinals 
$\mu \in N \cap \kappa$, $\cf(\sup(N \cap \lambda)) = \omega$, 
and for all regular uncountable cardinals 
$\nu \in N \cup \{ \theta \}$ different from 
$\lambda$, $\cf(\sup(N \cap \nu)) = \omega_1$.

\bigskip

We end the paper with several questions.
\begin{enumerate}
\item Does the existence of stationarily many indestructibly weakly guessing models which are 
not internally unbounded follow from Martin's maximum alone?
\item Does $\textsf{wGMP}$ imply $\textsf{GMP}$, or $\textsf{wIGMP}$ imply $\textsf{IGMP}$?
\item Viale \cite{viale} proved that the existence of stationarily many 
$\omega_1$-guessing models which are internally unbounded implies 
\textsf{SCH}. Does \textsf{SCH} follow from the existence of stationarily 
many weakly guessing models which are internally unbounded?
\item Is it consistent that there exists a forcing poset which has the 
$\omega_1$-approximation property, but does not have the countable 
covering property?
\item Is it consistent that Namba forcing on $\omega_2$ has 
the weak $\omega_1$-approximation property?
\end{enumerate}

\bibliographystyle{plain}
\bibliography{paper30}

\begin{thebibliography}{10}

\bibitem{jk26}
S.~Cox and J.~Krueger.
\newblock Quotients of strongly proper forcings and guessing models.
\newblock {\em J. Symbolic Logic}, 81(1):264--283, 2016.

\bibitem{jk28}
S.~Cox and J.~Krueger.
\newblock Indestructible guessing models and the continuum.
\newblock {\em Fundamenta Mathematicae}, 239:221--258, 2017.

\bibitem{cummings}
J.~Cummings.
\newblock Iterated forcing and elementary embeddings.
\newblock In {\em Handbook of set theory. Vols. 1, 2, 3}, pages 775--883.
  Springer, Dordrecht, 2010.

\bibitem{foreman}
M.~Foreman.
\newblock Ideals and generic elementary embeddings.
\newblock In {\em Handbook of set theory. Vols. 1, 2, 3}, pages 885--1147.
  Springer, Dordrecht, 2010.

\bibitem{fms}
M.~Foreman, M.~Magidor, and S.~Shelah.
\newblock Martin's maximum, saturated ideals and non-regular ultrafilters {I}.
\newblock {\em Ann. of Math.}, 127(1):1--47, 1988.

\bibitem{hamkins}
J.~Hamkins.
\newblock Extensions with the approximation and cover properties have no new
  large cardinals.
\newblock {\em Fund. Math.}, 180(3):257--277, 2003.

\bibitem{jk8}
J.~Krueger.
\newblock Internally club and approachable.
\newblock {\em Adv. Math.}, 213(2):734--740, 2007.

\bibitem{larson}
P.~Larson.
\newblock Separating stationary reflection principles.
\newblock {\em J. Symbolic Logic}, 65(1):247--258, 2000.

\bibitem{mitchell2}
W.~Mitchell.
\newblock On the {H}amkins approximation property.
\newblock {\em Ann. Pure Appl. Logic}, 144(1-3):126--129, 2006.

\bibitem{namba}
K.~Namba.
\newblock Independence proof of $(\omega,\omega_\alpha)$-distributivity law in
  complete {B}oolean algebras.
\newblock {\em Comment. Math. Univ. St. Paul.}, 19:1--12, 1970.

\bibitem{shelah}
S.~Shelah.
\newblock {\em Proper and Improper Forcing}.
\newblock Perspectives in Mathematical Logic. Springer-Verlag, Berlin, second
  edition, 1998.

\bibitem{todorcevic}
S.~{Todor\v cevi\' c}.
\newblock Some combinatorial properties of trees.
\newblock {\em Bull. London Math. Soc.}, 14(3):213--217, 1982.

\bibitem{viale}
M.~Viale.
\newblock Guessing models and generalized {L}aver diamond.
\newblock {\em Ann. Pure Appl. Logic}, 163(11):1660--1678, 2012.

\bibitem{vialeweiss}
M.~Viale and C.~Weiss.
\newblock On the consistency strength of the proper forcing axiom.
\newblock {\em Adv. Math.}, 228(5):2672--2687, 2011.

\bibitem{weiss}
C.~Weiss.
\newblock The combinatorial essence of supercompactness.
\newblock {\em Ann. Pure Appl. Logic}, 163(11):1710--1717, 2012.

\bibitem{woodin}
W.H. Woodin.
\newblock {\em The Axiom of Determinacy, Forcing Axioms, and the Non-Stationary
  Ideal}.
\newblock Walter de Gruyter, 1999.

\end{thebibliography}

\end{document}